\documentclass{amsart}
\usepackage[margin=1in]{geometry}
\usepackage{url}
\usepackage{amsmath,amssymb}
\usepackage{amsthm}
\usepackage{hyperref}
\usepackage{stmaryrd}
\usepackage{siunitx}
\usepackage{commath}
\usepackage{graphicx}
\usepackage{subcaption}

\usepackage{enumerate}
\usepackage{bm}
\usepackage{lipsum}
\usepackage{amsfonts}
\usepackage{epstopdf}
\usepackage{algorithmic}
\usepackage{amsopn}
\usepackage{mathtools}
\usepackage{framed,multirow}
\usepackage{booktabs}
\usepackage{tikz}
\usepackage{enumitem}
\usetikzlibrary{math}
\usepackage{hyperref}
\hypersetup{
	colorlinks   = true, 
	urlcolor     = blue, 
	linkcolor    = blue, 
	citecolor   = red 
}

\usepackage[noabbrev]{cleveref}
\numberwithin{equation}{section}
\newtheorem{theorem}{Theorem}[section]
\newtheorem{lemma}[theorem]{Lemma}

\theoremstyle{definition}

\theoremstyle{remark}
\newtheorem{remark}[theorem]{Remark}

\makeatletter
\newcommand{\tnorm}{\@ifstar\@tnorms\@tnorm}
\newcommand{\@tnorms}[1]{%
	\left|\mkern-1.5mu\left|\mkern-1.5mu\left|
	#1
	\right|\mkern-1.5mu\right|\mkern-1.5mu\right|
}
\newcommand{\@tnorm}[2][]{%
	\mathopen{#1|\mkern-1.5mu#1|\mkern-1.5mu#1|}
	#2
	\mathclose{#1|\mkern-1.5mu#1|\mkern-1.5mu#1|}
}
\makeatother

\makeatletter
\def\ps@pprintTitle{%
  \let\@oddhead\@empty
  \let\@evenhead\@empty
  \let\@oddfoot\@empty
  \let\@evenfoot\@oddfoot
}
\makeatother

\newcommand{\vertiii}[2][1]{\abs[#1]{\kern0.15ex\norm[#1]{#2}\kern-0.25ex}}

\usepackage{xspace,color}

\title[Symplectic second-order method for two-phase flow in porous media]{Discrete energy balance equation via a symplectic second-order method for two-phase flow in porous media}

\author[G. Sosa Jones]{Giselle Sosa Jones}
\address{Department of Mathematics and Statistics, Oakland University, 146 Library Drive, Rochester, MI 48309, USA.}
\email{gsosajones@oakland.edu}

\author[C. Trenchea]{Catalin Trenchea}
\address{Department of Mathematics, University of Pittsburgh, 301 Thackeray Hall, Pittsburgh, PA 15260, USA.}
\email{trenchea@pitt.edu}

\begin{document}

\begin{abstract}
We propose and analyze \normalcolor
a second-order partitioned time-stepping method for a 
two-phase flow problem in porous media. 
The algorithm is based on a refactorization of Cauchy's one-leg $\theta$-method. The main part consists of the implicit backward Euler method on $[t^n, t^{n+\theta}]$,
while part two uses a linear extrapolation on $[t^{n+\theta},t^{n+1}]$ to obtain the solution at $t^{n+1}$, equivalent to the forward Euler method. 

In the backward Euler step, the decoupled equations are solved iteratively. We prove that the iterations converge linearly to the solution of the coupled problem, under some conditions on the data. When $\theta = 1/2$, the algorithm is equivalent to the symplectic midpoint method. In the absence of the chain rule for time-discrete setting, we approximate the change in the free energy by the product of a second-order accurate discrete gradient (chemical potential) and the one-step increment of the state variables.
Similar to the continuous case, we also prove a discrete Gibbs free energy balance equation, without numerical dissipation. In the numerical tests we compare this implicit midpoint method with the 
classic backward Euler method, and two implicit-explicit 
time-lagging schemes. The midpoint method outperforms the other schemes in terms of rates of convergence, long-time behavior and energy approximation, for small and large values of the time step.
\end{abstract}

\maketitle





	
		
	
	
	\section{Introduction}
	\label{sec:introduction}
	Modeling 
 two-phase flow in porous media is of  importance in various fields such as petroleum engineering, hydrogeology, and environmental science. 
 Understanding the behavior of fluids, 
 for instance water and oil, within porous media is crucial for 
 accurately 
 \normalcolor
 predicting and optimizing reservoir performance, groundwater contamination, and pollutant transport.
	
	The flow of two immiscible components in a porous material is described by mass conservation of each component, together with Darcy's law for the velocity of each fluid. 
 These are complemented by two algebraic relations: a total mass conservation (`no-voids') constraint, and a relation between the difference in the phase pressures and the capillary pressure.
%
The mathematical model is a time-dependent, coupled system of partial differential equations (PDEs), 
or more precisely a differential algebraic equation (DAE) 
\cite{MR1009546,MR821044,MR1005520,MR749366,MR1138475,MR1893418,MR4586824,MR4328554},
with strong nonlinearities.
\\
At the continuous level, the DAE 
describing the two-phase flow in porous media satisfies a structural or geometric property, i.e., an energy balance equation holds, which can also be used to define a continuous-time Dirac structure \cite{MR4586824,10.5555/553988}.
 This means that the increment in the Hamiltonian (storage function)  is equal to the sum of the energy dissipation rate and the energy exchange (supply function) \cite{MR4586824,MR4001128}.

 From the numerical 
 analysis 
 viewpoint, 
 significant effort has been 
 devoted to designing and developing spatial discretizations for the solution of these problems; see, e.g., \cite{DongRiviere2016,RankinRiviere2015,Cappanera:2019,chen2001error} for finite element methods, \cite{douglas1983finite} for finite difference methods, and \cite{Eymard2003,ohlberger1997convergence,Michel03} for finite volume methods. 
 However, the literature on accurate and stable time-stepping methods is 
 limited. 
 Given that underground flow occurs over long periods of time, for accurate predictions it is paramount to employ time-stepping algorithms that remain stable for long simulation times. 
 In this sense, the development of 
 symplectic time-marching algorithms is important, since such methods preserve physical properties, guarantee stability and accuracy, and enable reliable long-term simulations \cite{MR2132573,MR2840298}. 
 The symplectic 
 integrators \normalcolor 
 which are derived from discrete versions of Hamilton's principle are called variational integrators and 
 are effective for conservative and dissipative or forced mechanical systems \cite{MR4001128,Lew2004AnOO}.
 \normalcolor
 
 In \cite{kou2022energy} an implicit-explicit (IMEX), first-order accurate in time, energy stable method for two-phase flow in porous media is introduced. 
 The method is based on an invariant energy quadratization approach, 
 using a regularized free energy, 
 and 
 relaxing the requirement on the maximum principle of saturations. The authors prove the energy stability of the scheme, 
 with nonzero numerical dissipation being introduced by the time-stepping method.
	
	In this paper, we present a symplectic second-order, energy-stable partitioned time-stepping method for the two-phase flow problem in a porous medium. 
 %
 Since in the discrete setting the chain rule is not valid, we approximate the change of in the free energy 
 by the product of a (second-order accurate) discrete gradient of the energy and the one-step increment of the state variables.
 Hence, at the discrete time level, 
 we 
 prove a discrete energy balance equation, which mimics the energy balance relation that the continuous problem satisfies.
 The preservation of the discrete-time Dirac structure 
 is based on 
 the symplectic time integration, 
 yielding accurate long time simulations.
 The algorithm is based on the refactorized implementation of midpoint method: 
 to obtain the solution at time $t^{n+1}$, first an implicit solver is used to obtain the solution at time $t^{n+\theta}$, with $\theta \in [0,1]$, and then an explicit method is used to calculate the solution at time $t^{n+1}$ \cite{MR4092601,MR4265875,MR4342399,MR4505477,MR4398357}. 
 Due to the nonlinearity of the problem, the implicit step requires a subiteration scheme, which we show that it converges linearly, under certain regularity conditions on the problem parameters and the smoothness of the exact solution. 
 %
    Traditionally, the numerical methods for multiphase flow in porous media use either a time lagging IMEX scheme or a backward Euler method for the time discretization, see e.g., \cite{CappaneraRiviere,EpshteynRiviere2009,MR4446767}. 
In the numerical tests 
we compare the performance of the 
proposed method versus the backward Euler method, a first-order, and a second-order time lagging scheme.
	
The rest of the paper is organized as follows. 
In \Cref{sec:prob_description} we briefly present
the physical problem, and in \Cref{sec:time_stepping_method} we introduce the partitioned time-stepping method.
%
In \Cref{sec:convergence_proof} 
we 
demonstrate
the linear convergence of the iterations from the implicit partitioned part of the algorithm, and energy stability properties of the method. 
We 
prove a continuous and a discrete-time energy balance equation, which imply dissipation of the free energy and stability of the numerical method.  Finally, numerical results are given in \Cref{sec:numerical_results}.
	
	\section{Problem Description}
	\label{sec:prob_description}
	Consider a reservoir $\Omega \subset \mathbb{R}^d$, $d=2,3$, in which there are two immiscible components: water and oil, 
 that are in aqueous ($a$) and liquid ($\ell$) phases, respectively. Disregarding the gravity effects, and using Darcy's law for the velocity of each phase, the conservation of mass for each component states that:
	\begin{subequations}
		\begin{align}
			\phi \partial_t s_j + \nabla \cdot \boldsymbol{u}_j &= q_j
			\\
			\boldsymbol{u}_j & = -\kappa \lambda_j \nabla p_j,
		\end{align}
		\label{eq:mass_cons}
	\end{subequations}
	for $j = a,\ell$. Here, $\phi$ and $\kappa$  denote the porosity and the absolute permeability of the medium, respectively. 
 Furthermore, $s_j$, $\lambda_j$ and $p_j$ are the saturation, mobility and pressure of phase $j$, respectively. 
 Finally, $q_j$ represent the source/sink terms. 
 We 
 assume that, while $\phi$ is a constant, $\kappa$ is a piecewise constant function in space, and the source terms are given functions of time and space. 
 \\
 Given that the reservoir is fully saturated with the two components, we have the algebraic relation
	\begin{equation}
		s_a + s_\ell = 1.
		\label{eq:saturations_relation}
	\end{equation}
	Moreover, see e.g. \cite{brooks-corey-1964_Hppm}, the difference in phase pressures is given by the capillary pressure $p_{c}$, i.e.,
	\begin{equation}
		p_\ell - p_a = p_{c}.
		\label{eq:algebraic_relations}
	\end{equation}
 The equations \eqref{eq:mass_cons}-\eqref{eq:algebraic_relations} represent a differential-algebraic system of equations (DAE) \cite{MR2657217,MR1638643}. 
	We consider a formulation where the primary unknowns are the pressure of the liquid phase, $p_\ell$, and the aqueous saturation $s_a$. 
 The phase mobilities and the capillary pressure depend on the primary unknowns:
	\begin{equation*}
		\lambda_j \equiv \lambda_j(s_a), \quad p_{c} \equiv p_{c}(s_a).
	\end{equation*}
	Since $p_{c}$ is a decreasing function of $s_a$, we have $p_{c}' < 0$ \cite{ChenBook,BearBook}. The PDE for $p_\ell$ is obtained by adding up the mass conservation equations 
 \eqref{eq:mass_cons}, whereas the equation for $s_a$ is given by \cref{eq:mass_cons}. After using the algebraic relations in \eqref{eq:saturations_relation} and \eqref{eq:algebraic_relations}, and substituting the velocities $\boldsymbol{u}_j$, we obtain the following system of PDEs:
	\begin{subequations}
		\begin{align}
			-\nabla \cdot \left(\lambda \kappa \nabla p_\ell\right) + \nabla \cdot \left(\lambda_a \kappa \nabla p_{c}\right) &= q,
			\label{eq:pl_eq}
			\\
			\phi \partial_t s_a + \nabla \cdot \left(\lambda_a \kappa p_{c}' \nabla s_a\right) - \nabla \cdot \left(\lambda_a \kappa \nabla p_\ell\right) &= q_a,
			\label{eq:sa_eq}
		\end{align}
		\label{eq:PDE_system}
	\end{subequations}
 where $\lambda = \lambda_\ell + \lambda_a$, and $q = q_\ell + q_a$. 
The system is 
endowed with the initial conditions $p_\ell|_{t = 0} = p_\ell^0$, $s_a|_{t = 0} = s_a^0$. Moreover, for the analysis in this paper we consider non-homogeneous Dirichlet boundary conditions on $\partial \Omega$. 

The Gibbs free energy $F$ in a system with two immiscible components can be written as a function of $s_a$ and $s_\ell$ as follows
	\begin{equation}
		F(s_a, s_\ell) = \sum_{j = a, \ell}\gamma_js_j\del{\ln(s_j) - 1} + \gamma_{a\ell}s_a s_\ell,
\label{eq:free_energy}
	\end{equation}
	where $\gamma_a, \gamma_\ell, \gamma_{a\ell}$ are the (constant) energy parameters \cite{atkins2014atkins,kou2022energy,MR4446767}. 
	The chemical potentials $\mathcal{P}_j$, for $j=\ell,a$ are defined as $\mathcal{P}_j = \partial_{s_j}F$. The difference in pressures is equal to the difference in potentials \cite{GaoPorousMedia,ChenBook,atkins2014atkins}, i.e.,
\begin{align}
\label{eq:chemicalpotential-capillarypressure}
\nu_{a} \equiv \mathcal{P}_a - \mathcal{P}_\ell = p_a - p_\ell =-p_{c}.
\end{align}
 This gives us a relation between the chemical potentials and the phase pressures. Noting that $s_\ell = 1-s_a$ and using the chain rule, the relation between the free energy and the capillary pressure is given by
	\begin{align}
		\partial_t F(s_a, s_\ell) &= \partial_{s_a} F(s_a, s_\ell) \partial_t s_a + \partial_{s_\ell}F(s_a, s_\ell) \partial_t s_\ell 
  = \mathcal{P}_a \partial_t s_a  - \mathcal{P}_\ell \partial_t s_a
		\nonumber
  = \nu_{a} \partial_t s_a
		\nonumber
		\\
		&= -p_{c}\partial_t s_a.
		\label{eq:chain_rule}
	\end{align}
Using the definition of the Gibbs free energy \eqref{eq:free_energy}, we have an explicit expression for $\nu_a$ given by
	\begin{align}
		\nu_a &= \mathcal{P}_a - \mathcal{P}_\ell
  = \partial_{s_a}F(s_a,s_\ell) - \partial_{s_\ell}F(s_a,s_\ell)
  = \gamma_a \ln(s_a) - \gamma_\ell \ln(1-s_a) + \gamma_{a\ell}(1-2s_a).
  \label{eq:nu_a_definition_continuous}
	\end{align}
 The equations \eqref{eq:chemicalpotential-capillarypressure}, \eqref{eq:chain_rule}, and \eqref{eq:nu_a_definition_continuous} will be used in \Cref{sec:convergence_proof} to prove an energy balance equation in the continuous and discrete time cases.
	\section{Time-stepping method}
	\label{sec:time_stepping_method}
	
	The time interval $[0,T]$ of interest is divided into time levels $t^0 = 0 < t^1 < \cdots < t^N = T$, where $t^{n+1} = t^n + \tau^n$. Let $t^{n+\theta^n} = t^n + \theta^n\tau^n$, for any $\theta^n \in [
 0,1]$, $\forall n \geq 0$. For any time-dependent 
 variable  $g(t)$, $g^n$ denotes an approximation to $g(t^n)$. 
 The time-stepping method is based on the refactorization of the (implicit) midpoint method, where we first use the implicit (backward) Euler method to approximate the solution at $t^{n+\theta^n}$, and then use the explicit (forward) Euler method to obtain the solution at $t^{n+1}$, see e.g., \cite{MR4092601}. Since the problem is nonlinear, it requires a subiterative procedure to solve the system at $t^{n+\theta^n}$. For applications involving computations on large physical domains, we partition the computations in a sequential manner, i.e., we first solve for the pressure of the liquid phase $p_\ell$, then use this approximation in the equation for aqueous saturation $s_a$, and iterate until convergence.

%
%
	Let $s^{n+\theta^n}_{a_{(i)}}$ and $p^{n+\theta^n}_{\ell_{(i)}}$ denote $s_a^{n+\theta^n}$ and $p_\ell^{n+\theta^n}$, respectively, at 
 the index $i$ of the subiteration. 
 For the sake of simplifying the presentation, we denote $g^{n+\theta^n}_{(i)} = g(s^{n+\theta^n}_{a_{(i)}})$, where $g$ is any coefficient that depends on $s_a$. First, we need to calculate $p_\ell^1$ and $s_a^1$ using a second-order method, for which a monolithic method could be used. 
 
	\noindent \textbf{Step 1}: First, we 
 define the initial guess $p_{\ell_{(0)}}^{n+\theta^n}$, $s_{a_{(0)}}^{n+\theta^n}$ 
 as the  extrapolated values:
	\begin{subequations}
		\begin{align*}
			p_{\ell_{(0)}}^{n+\theta^n} &= \frac{1+\theta^n\tau^n + \tau^{n-1} - \theta^{n-2}\tau^{n-2}}{1 + \theta^{n-1}\tau^{n-1} - \theta^{n-2}\tau^{n-2}}p_\ell^{n+\theta^{n-1}-1} - \frac{\del[1]{1-\theta^{n-1}}\tau^{n-1} + \theta^n\tau^n}{1+\theta^{n-1}\tau^{n-1} - \theta^{n-2}\tau^{n-2}}p_\ell^{n+\theta^{n-2} - 2},
			\\
			s_{a_{(0)}}^{n+\theta^n} &= \del{1 + \frac{\theta^n\tau^n}{\tau^{n-1}}}s_{a}^n - \frac{\theta^n\tau^n}{\tau^{n-1}}s_a^{n-1},
		\end{align*}
		\label{eq:initial_guess}
	\end{subequations}
	
	
	\noindent \textbf{Step 2}: For every $i \geq 0$ until convergence, we solve the following equations sequentially, which correspond to a backward Euler method to obtain $p_{\ell_{(i+1)}}^{n+\theta^n}$ and $s_{a_{(i+1)}}^{n+\theta^n}$:
	\begin{subequations}
 \label{eq:BE-iteration}
		\begin{align}
			-\nabla \cdot \left(\lambda_{{(i)}}^{n+\theta^n} \kappa \nabla p_{\ell_{(i+1)}}^{n+\theta^n}\right) &= q^{n+\theta^n} - \nabla \cdot \left(\lambda_{a_{(i)}}^{n+\theta^n} \kappa \nabla p_{c_{(i)}}^{n+\theta^n}\right)
\\
			\phi \frac{s_{a_{(i+1)}}^{n+\theta^n} - s_a^n}{\theta^n\tau^n} + \nabla \cdot \left(\lambda_{a_{(i)}}^{n+\theta^n} \kappa p'^{n+\theta^n}_{c_{(i)}} \nabla s_{a_{(i+1)}}^{n+\theta^n}\right)  &= q_a^{n+\theta^n} + \nabla \cdot \left(\lambda_{a_{(i)}}^{n+\theta^n} \kappa \nabla p_{\ell_{(i+1)}}^{n+\theta^n}\right)
		\end{align}
		\label{eq:system_iterations}
	\end{subequations}
	%
	When convergence of the sequences $\{ p_{\ell_{(i+1)}}^{n+\theta^n}\}_{i} $, $\{s_{a_{(i+1)}}^{n+\theta^n} \}_i$ is attained, the 
 limit values $p_\ell^{n+\theta^n}$, $s_a^{n+\theta^n}$ satisfy the following coupled system:
	\begin{subequations}
		\begin{align}
			-\nabla \cdot \left(\lambda^{n+\theta^n} \kappa  \nabla p_\ell^{n+\theta^n}\right)+ \nabla \cdot \left(\lambda_a^{n+\theta^n} \kappa \nabla p_{c}^{n+\theta^n}\right) &= q^{n+\theta^n} ,
			\label{eq:BE_pl}
			\\
			\phi \frac{s_a^{n+\theta^n} - s_a^n}{\theta^n \tau^n} + \nabla \cdot \left(\lambda_a^{n+\theta^n} \kappa p_c'^{n+\theta^n} \nabla s_a^{n+\theta^n}\right) - \nabla \cdot \left(\lambda_a^{n+\theta^n} \kappa \nabla p_\ell^{n+\theta^n}\right) &= q_a^{n+\theta^n}.
			\label{eq:BE_sa}
		\end{align}	
		\label{eq:system_midstep}
	\end{subequations}

	\noindent \textbf{Step 3}: Obtain the solution at $t^{n+1}$ by evaluating the following extrapolation:
	\begin{equation}
			s_a^{n+1} = \frac{1}{\theta^n}s_a^{n+\theta^n} - \frac{1-\theta^n}{\theta^n}s_a^n,
		\label{eq:forward_step}
	\end{equation}
	which is equivalent to solving the 
 forward Euler problem:
		\begin{equation}
			\phi \frac{s_a^{n+1} - s_a^{n+\theta^n}}{(1-\theta^n) \tau^n} = q_a^{n+\theta^n} -  \nabla \cdot \left(\lambda_a^{n+\theta^n} \kappa p_c'^{n+\theta^n} \nabla s_a^{n+\theta^n}\right) + \nabla \cdot \left(\lambda_a^{n+\theta^n} \kappa \nabla p_\ell^{n+\theta^n}\right).
			\label{eq:FE_sa}
		\end{equation}	
		\label{eq:system_finalstep}
We note that the solution to \eqref{eq:BE_pl}--\eqref{eq:FE_sa} solves
 \begin{align}
 \label{eq:midpoint}
 \begin{array}{rr}
\displaystyle
-\nabla \cdot \left(\lambda^{n+\theta^n} \kappa  \nabla p_\ell^{n+\theta^n}\right)+ \nabla \cdot \left(\lambda_a^{n+\theta^n} \kappa \nabla p_{c}^{n+\theta^n}\right) 
= q^{n+\theta^n}
\vspace{0.1cm}\\
\displaystyle
\phi \frac{s_a^{n+1} - s_a^{n}}{\tau^n} 
+ \nabla \cdot \left(\lambda_a^{n+\theta^n} \kappa p_c'^{n+\theta^n} \nabla s_a^{n+\theta^n}\right)
-
\nabla \cdot \left(\lambda_a^{n+\theta^n} \kappa \nabla p_\ell^{n+\theta^n}\right)
= q_a^{n+\theta^n}  .
\end{array}
\end{align}	
	\section{Convergence and energy stability}
	\label{sec:convergence_proof}
	
	In this section, we show the convergence of the iterations in the backward Euler step \cref{eq:system_iterations}, and an energy balance relation, which yields the stability (energy dissipation inequality \cite{MR4586824}) 
 of the time discretization method proposed in this work.
	
	\subsection{Notation and hypotheses}
	
	In what follows, we denote by $(\cdot,\cdot)_D$ the standard $L^2$ inner product in a domain $D$. Moreover, $\norm[1]{\cdot}_D$ denotes the standard $L^2$ norm in $D$.  We make the following assumptions on the coefficients and the analytical solution of the problem.
	
\begin{enumerate}
	\item \label{A1}
	$0 < c_{\lambda_j} \leq \lambda_j \leq C_{\lambda_j}$, for $j=\ell, a$.
	\item \label{A2}
	$0 < c_{\lambda} \leq \lambda \leq C_{\lambda}$.
	\item \label{A3}
	$0 < \kappa_\ast \leq \kappa \leq \kappa^\ast$.
	\item \label{A4}
	$0 \leq c_{p_{c}} \leq -p_{c}' \leq C_{p_{c}}$.
	\item \label{A5}
	$0 \leq c_{\nabla p_{c}} \leq \norm[1]{\nabla p_{c}}_{L^\infty(\Omega)} \leq C_{\nabla p_{c}}$.
	\item \label{A6}
	$\abs{\lambda_j(s_{a_2}) - \lambda_j(s_{a_1})} \leq L_{\lambda_j}\abs{s_{a_2} - s_{a_1}}$, for $j=\ell, a$.
	\item \label{A7}
	$\abs{\lambda(s_{a_2}) - \lambda(s_{a_1})} \leq L_{\lambda} \abs{s_{a_2} - s_{a_1}}$.
	\item \label{A8}
	$\abs{p_{c}'(s_{a_2}) - p_{c}'(s_{a_1})} \leq L_{p_{c}}\abs{s_{a_2} - s_{a_1}}$.
	\item \label{A9}
	$\norm[1]{\nabla p_{c}(s_{a_2}) - \nabla p_{c}(s_{a_1})}_\Omega \leq L_{\nabla p_{c}}\norm[1]{s_{a_2} - s_{a_1}}_\Omega$.
	\item \label{A10}
	$p_\ell, \,s_a \in W^{1,\infty}(\Omega)$, so that $\norm[1]{\nabla p_\ell}_{L^\infty(\Omega)} = C^\ell_\infty < \infty$, $\norm[1]{\nabla s_a}_{L^\infty(\Omega)} = C^a_\infty < \infty$.
\end{enumerate}
	We remark that hypotheses 
\eqref{A5} and 
\eqref{A9} might appear to be restrictive, as the capillary pressure $p_c$ is a function of $s_a$.  However, such an assumption is widely used in the context of the analysis of multiphase flow in porous media, e.g., \cite{chen2001error,radu2018robust}. For instance, in \cite{chen2001error}, assumptions (A5) and (A7) state that the functions $\gamma_1$ and $\gamma_2$, which contain the gradient of the capillary pressure, are bounded and Lipschitz continuous with respect to the primary unknown $\theta$. 
	%
	
	\subsection{Convergence of the iterations}
	
	First we show that as the index increases $i 
 \nearrow \infty$, the sequences $\{p_{\ell_{(i+1)}}^{n+\theta^n}\}_{i\geq 0}$ and $\{s_{a_{(i+1)}}^{n+\theta^n}\}_{i\geq 0}$ converge strongly in $H^1(\Omega)$, respectively, to $p_{\ell}^{n+\theta^n}$ and $s_{a}^{n+\theta^n}$, 
 the solution of \eqref{eq:midpoint}. 
 To simplify the notation, we denote
	\begin{equation}
		K_\ell = \lambda \kappa, \quad K_a = -\lambda_a \kappa p_c', \quad B_a = \lambda_a \kappa.
		\label{eq:new_notaion}
	\end{equation}
%
We mention
 that assumptions 
 \eqref{A1}--\eqref{A4} imply that
\begin{enumerate}[resume]
		\item
  \label{A11}
  $0 < c_{K_j} \leq K_\ell^n \leq C_{K_j}$, for $j=\ell,a,$ and for all $n$.
		\item
  \label{A12}
  $0 < c_{B_a} \leq B_a^n \leq C_{B_a}$, for all $n$.
	\end{enumerate}
	Moreover, hypotheses 
 \eqref{A1}--\eqref{A8} imply that
	\begin{enumerate}[resume]
		\item
  \label{A13}
  $\abs{K_j(s_{a_2}) - K_j(s_{a_1})} \leq L_{K_j} \abs{s_{a_2} - s_{a_1}}$, for $j=\ell,a,v$.
		\item
  \label{A14}
  $\abs{B_a(s_{a_2}) - B_a(s_{a_1})} \leq L_{B_a} \abs{s_{a_2} - s_{a_1}}$.	
	\end{enumerate}
	
	Now, we are ready to prove the convergence of the iterations defined by \cref{eq:system_iterations}.
	\begin{lemma}
Assume that the time step $\tau^n$ is small enough so that it satisfies $\tilde{C} \theta^n \tau^n < 1$, where $\tilde{C}$ is a constant that depends on the hypotheses above. 
  Then the sequences $p_{\ell_{(i)}}^{n+\theta^n}$, $s_{a_{(i)}}^{n+\theta^n}$ calculated in \cref{eq:system_iterations} 
  converge  to $p_\ell^{n+\theta^n}$ and $s_a^{n+\theta^n}$, 
  the solution of \eqref{eq:midpoint}.
The convergence is linear, and is in the 
strong topology on $H^1(\Omega)$, namely
  \begin{equation*}
			\norm[1]{  s_a^{n+\theta^n} - s_{a_{(i+1)}}^{n+\theta^n} }_\Omega^2 + \theta^n \tau^n\norm[1]{\nabla \big( p_\ell^{n+\theta^n} - p_{\ell_{(i+1)}}^{n+\theta^n} \big) }_\Omega^2 
   + \theta^n \tau^n \norm[1]{\nabla  \big( s_a^{n+\theta^n} - s_{a_{(i+1)}}^{n+\theta^n} \big)}_\Omega^2
			\leq 
			\del[1]{\tilde{C} \theta^n \tau^n}^i \norm[1]{\delta_{(0)}^a}_\Omega^2.
		\end{equation*}
		%
		%
	\end{lemma}
	\begin{proof}
		Define 
  $\delta^\ell_{(i+1)} 
  = p_\ell^{n+\theta^n} - p_{\ell_{(i+1)}}^{n+\theta^n}$ 
  and 
  $\delta^a_{(i+1)} 
  = s_a^{n+\theta^n} - s_{a_{(i+1)}}^{n+\theta^n}$. We subtract \eqref{eq:system_iterations} from \eqref{eq:system_midstep}, and analyze each equation separately. 
  First, for the liquid pressure $p_\ell$, we have
		\begin{multline*}
			-\nabla \cdot \left(K_{\ell_{(i)}}^{n+\theta^n} \nabla \delta^\ell_{(i+1)}\right) - \nabla \cdot \left((K_\ell^{n+\theta^n} 
			- K_{\ell_{(i)}}^{n+\theta^n})\kappa \nabla p_\ell^{n+\theta^n}\right) 
			\\
			+ \nabla \cdot \left(\lambda_{a_{(i)}}^{n+\theta^n} \kappa \del{\nabla p_c^{n+\theta^n} - \nabla p_{c_{(i+1)}}^{n+\theta^n}}\right) 
		   + \nabla \cdot \left(\kappa(\lambda_a^{n+\theta^n} - \lambda_{a_{(i)}}^{n+\theta^n} )\nabla p_{c}^{n+\theta^n}\right) = 0,
		\end{multline*}
where we have used the notation 
\eqref{eq:new_notaion}.
		We note that this equation is complemented with homogeneous Dirichlet boundary conditions. We multiply by $\delta^\ell_{(i+1)}$, integrate over the domain $\Omega$ and perform integration by parts to obtain
		\begin{multline*}
			\int_\Omega K_{\ell_{(i)}}^{n+\theta^n} \abs{\nabla \delta^\ell_{(i+1)}}^2 \dif x + \int_\Omega (K_\ell^{n+\theta^n} 
			- K_{\ell_{(i)}}^{n+\theta^n}) \nabla p_\ell^{n+\theta^n} \cdot \nabla \delta^\ell_{(i+1)} \dif x
			\\
			- \int_\Omega \kappa \lambda_{a_{(i)}}^{n+\theta^n}  \del{\nabla p_c^{n+\theta^n} - \nabla p_{c_{(i+1)}}^{n+\theta^n}} \cdot \nabla \delta^\ell_{(i+1)} \dif x
			- \int_\Omega \kappa (\lambda_a^{n+\theta^n}  - \lambda_{a_{(i)}}^{n+\theta^n} )\nabla p_{c}^{n+\theta^n} \cdot \nabla \delta^\ell_{(i+1)} \dif x = 0.
		\end{multline*}
	%
		Rearranging terms, using hypotheses 
  \eqref{A1}, \eqref{A3}, \eqref{A11}, and the Cauchy-Schwarz inequality we obtain
		\begin{multline*}
			c_{K_\ell}  \norm[1]{\nabla \delta^\ell_{(i+1)}}_\Omega^2 
\leq  \int_\Omega (K_{\ell_{(i)}}^{n+\theta^n} - K_\ell^{n+\theta^n} )\nabla p_\ell^{n+\theta^n} \cdot \nabla \delta^\ell_{(i+1)} \dif x
\\
			+ \kappa^* C_{\lambda_a} \int_\Omega  \del{\nabla p_c^{n+\theta^n} - \nabla p_{c_{(i+1)}}^{n+\theta^n}} \cdot \nabla \delta^\ell_{(i+1)} \dif x
			+ \kappa^*  \int_\Omega  (\lambda_a^{n+\theta^n}  - \lambda_{a_{(i)}}^{n+\theta^n} )\nabla p_{c}^{n+\theta^n} \cdot \nabla \delta^\ell_{(i+1)} \dif x ,
		\end{multline*}
which by employing 
  hypotheses 
  \eqref{A6}, \eqref{A9} and \eqref{A13} 
  further implies
		%
		\begin{multline*}
			c_{K_\ell} \norm[1]{\nabla \delta^\ell_{(i+1)}}_\Omega^2 
			\leq L_{K_\ell} \norm[1]{\delta_{(i)}^a\nabla p_\ell^{n+\theta^n}}_\Omega \norm[1]{\nabla \delta^\ell_{(i+1)}}_\Omega + \kappa^* C_{\lambda_a} L_{\nabla p_c}\norm[1]{\delta_{(i)}^a}_\Omega \norm[1]{\nabla \delta^\ell_{(i+1)}}_\Omega
\\
			+ \kappa^* L_{\lambda_a} \norm[1]{\delta_{(i)}^a\nabla p_{c}^{n+\theta^n}}_\Omega \norm[1]{\nabla \delta^\ell_{(i+1)}}_\Omega.
\label{eq:error_pl_2}
		\end{multline*}
Now we utilize hypotheses \eqref{A5}, \eqref{A10} 
%
and Young's inequality 
to get
		\begin{equation}
			\norm[1]{\nabla \delta^\ell_{(i+1)}}_\Omega^2 
			\leq 
			\del{\frac{L_{K_\ell} C^\ell_\infty + \kappa^* C_{\lambda_a} L_{\nabla p_c} + \kappa^* L_{\lambda_a}C_{\nabla p_c}}{c_{K_\ell}}}^2 \norm[1]{\delta^a_{(i)}}_\Omega^2.
\label{eq:final_error_pl}
		\end{equation}
		Next we work on the aqueous saturation $s_a$. Proceeding as for the liquid pressure, we obtain
		\begin{multline*}
			\phi \frac{\delta^a_{(i+1)}}{\theta^n \tau^n} - \nabla \cdot \left(K_{a_{(i)}}^{n+\theta^n} \nabla \delta^{a}_{(i+1)}\right) - \nabla \cdot \left((K_{a}^{n+\theta^n} - K_{a_{(i)}}^{n+\theta^n}) \nabla s_a^{n+\theta^n}\right) 
\\
			- \nabla \cdot \left(B_{a_{(i)}}^{n+\theta^n} \nabla \delta^\ell_{(i+1)}\right) - \nabla \cdot \left((B_a^{n+\theta^n} - B_{a_{(i)}}^{n+\theta^n}) \nabla p_\ell^{n+\theta^n}\right) = 0.
		\end{multline*}
		This equation is complemented also with homogeneous Dirichlet boundary conditions. We multiply by $\delta^a_{(i+1)}$, integrate over $\Omega$ and perform integration by parts, to obtain
		\begin{multline*}
			\frac{\phi}{\theta^n \tau^n}\norm[1]{\delta^a_{(i+1)}}_\Omega^2 + \int_\Omega K_{a_{(i)}}^{n+\theta^n}  \abs{\nabla \delta^{a}_{(i+1)}}^2 \dif x + \int_\Omega (K_{a}^{n+\theta^n} - K_{a_{(i)}}^{n+\theta^n}) \nabla s_a^{n+\theta^n} \cdot \nabla \delta^{a}_{(i+1)} \dif x
\\
			+ \int_\Omega B_{a_{(i)}}^{n+\theta^n} \nabla \delta^\ell_{(i+1)} \cdot \nabla \delta^{a}_{(i+1)} \dif x + \int_\Omega (B_a^{n+\theta^n} - B_{a_{(i)}}^{n+\theta^n}) \nabla p_\ell^{n+\theta^n} \cdot \nabla \delta^{a}_{(i+1)} \dif x = 0.
		\end{multline*}
		Rearranging terms and using hypotheses 
  \eqref{A9}, \eqref{A10}, and \eqref{A12}, and the Cauchy-Schwarz inequality, we get
		\begin{multline*}
			\frac{\phi}{\theta^n \tau^n}\norm[1]{\delta^a_{(i+1)}}_\Omega^2 + c_{K_a}  \norm[1]{\nabla \delta^{a}_{(i+1)}}_\Omega^2
			\leq 
			\int_\Omega (K_{a}^{n+\theta^n} - K_{a_{(i)}}^{n+\theta^n}) \nabla s_a^{n+\theta^n} \cdot \nabla \delta^{a}_{(i+1)} \dif x
			\\
			+ C_{B_a} \norm[1]{\nabla \delta^\ell_{(i+1)}}_\Omega \norm[1]{\nabla \delta^{a}_{(i+1)}}_\Omega
			+  L_{B_a} \norm[1]{\delta^a_{(i)}\nabla p_\ell^{n+\theta^n}}_\Omega \norm[1]{\nabla \delta^{a}_{(i+1)}}_\Omega.
		\end{multline*}
		Thanks to hypotheses 
  \eqref{A11}, we 
  have $K_{a}^{n+\theta^n} - K_{a_{(k,k)}}^{n+\theta^n} \leq L_{K_a}\abs[1]{\delta_{(i)}^a}$,
		%
and	thus,
		\begin{multline*}
			\frac{\phi}{\theta^n \tau^n}\norm[1]{\delta^a_{(i+1)}}_\Omega^2 + c_{K_a}  \norm[1]{\nabla \delta^{a}_{(i+1)}}_\Omega^2
			\leq 
			L_{K_a} \norm[1]{\delta^a_{(i)} \nabla s_a^{n+\theta^n}}_\Omega \norm[1]{\nabla \delta^{a}_{(i+1)}}_\Omega 
			\\
			+ C_{B_a} \norm[1]{\nabla \delta^\ell_{(i+1)}}_\Omega \norm[1]{\nabla \delta^{a}_{(i+1)}}_\Omega
			+ L_{B_a} \norm[1]{\delta^a_{(i)} \nabla p_\ell^{n+\theta^n}}_\Omega \norm[1]{\nabla \delta^{a}_{(i+1)}}_\Omega.
		\end{multline*}
		Using hypothesis 
  \eqref{A8} and grouping terms, we obtain
		\begin{equation*}
			\frac{\phi}{\theta^n \tau^n}\norm[1]{\delta^a_{(i+1)}}_\Omega^2 + c_{K_a} \norm[1]{\nabla \delta^{a}_{(i+1)}}_\Omega^2
			\leq 
			\del{L_{K_{a}} C_\infty^a +  L_{B_a}C_\infty^\ell} \norm[1]{\delta_{(i)}^a}_\Omega \norm[1]{\nabla \delta^{a}_{(i+1)}}_\Omega 
			+ C_{B_a}  \norm[1]{\nabla \delta^\ell_{(i+1)}}_\Omega \norm[1]{\nabla \delta^{a}_{(i+1)}}_\Omega,
		\end{equation*}
which by using 
Young's inequality 
yields
		\begin{equation*}
			\frac{2\phi}{\theta^n \tau^n}\norm[1]{\delta^a_{(i+1)}}_\Omega^2 + c_{K_a} \norm[1]{\nabla \delta^{a}_{(i+1)}}_\Omega^2
			\leq 
			\frac{2}{c_{K_a}} \del{L_{K_{a}} C_\infty^a +  L_{B_a}C_\infty^\ell}^2 \norm[1]{\delta_{(i)}^a}_\Omega^2
			+ \frac{2\del[0]{C_{B_a} }^2}{c_{K_a} } \norm[1]{\nabla \delta^\ell_{(i+1)}}_\Omega^2.
		\end{equation*}
		Finally, we substitute in \eqref{eq:final_error_pl}, the error estimate for $p_\ell$, 
  to get the 
  expression for the error in $s_a$:
		\begin{equation}
			\frac{2\phi}{\theta^n \tau^n}\norm[1]{\delta^a_{(i+1)}}_\Omega^2 +  c_{K_a}  \norm[1]{\nabla \delta^{a}_{(i+1)}}_\Omega^2
			\leq 
			C_a \norm[1]{\delta_{(i)}^a}_\Omega^2,
\label{eq:final_error_sa}
		\end{equation}
		where the constant $C_a$ only depends on the coefficients appearing in the equations.
Adding up 
\eqref{eq:final_error_pl} and \eqref{eq:final_error_sa} we obtain
		\begin{equation}
			\frac{2\phi}{\theta^n \tau^n}\norm[1]{\delta^a_{(i+1)}}_\Omega^2 + \norm[1]{\nabla \delta^\ell_{(i+1)}}_\Omega^2 +  c_{K_a} \norm[1]{\nabla \delta^{a}_{(i+1)}}_\Omega^2
			\leq 
			C \norm[1]{\delta_{(i)}^a}_\Omega^2,
\label{eq:estimate1}
		\end{equation}
		where the constant $C$ only depend on the coefficients appearing in the equations. We can rewrite this more conveniently as
		\begin{equation*}
			\norm[1]{\delta^a_{(i+1)}}_\Omega^2 + \theta^n \tau^n\norm[1]{\nabla \delta^\ell_{(i+1)}}_\Omega^2 + \theta^n \tau^n \norm[1]{\nabla \delta^{a}_{(i+1)}}_\Omega^2
			\leq 
			\tilde{C} \theta^n \tau^n \norm[1]{\delta_{(i)}^a}_\Omega^2.
		\end{equation*}
		Applying this relation recursively, we obtain that
		\begin{equation*}
			\norm[1]{\delta^a_{(i)}}_\Omega^2 + \theta^n \tau^n\norm[1]{\nabla \delta^\ell_{(i)}}_\Omega^2 + \theta^n \tau^n \norm[1]{\nabla \delta^{a}_{(i)}}_\Omega^2
			\leq 
			\del[1]{\tilde{C} \theta^n \tau^n}^i \norm[1]{\delta_{(0)}^a}_\Omega^2,
		\end{equation*}
		for any $i \geq 1$. If $\tilde{C} \theta^n \tau^n < 1$, then the $\norm[1]{\nabla \delta^j_{(i)}} \rightarrow 0$ as $i \rightarrow \infty$, for $j=\ell,a$.
	\end{proof}
	\begin{remark}
		We note that $\norm[1]{\nabla \delta^\ell_{(i)}}_\Omega^2$ does not appear on the right hand side of \cref{eq:estimate1}. This is because the equation for $s_a$ in \eqref{eq:system_iterations} uses the liquid pressure at the current iteration instead of the previous one. 
	\end{remark}
	
	\subsection{Energy balance equation and energy
 stability}

First, we will use the relation between the chemical potentials and the capillary pressure, \cref{eq:chemicalpotential-capillarypressure}, to show an energy balance equation at the continuous time level. In particular, the strict dissipativity \cite[pp. 422]{MR4586824} of the DAE system \eqref{eq:PDE_system}.
The following results rely on the relation \cref{eq:chemicalpotential-capillarypressure} which relates the gradient of the free energy with the capillary pressure. In the thermodynamically consistent case of the Gibbs free energy in \cref{eq:free_energy}, $p_c$ has to be chosen according to \cref{eq:nu_a_definition_continuous}. For different capillary pressure models (e.g., Brooks--Corey), the following results would still hold, but the free energy of the system would be different.
	
	\begin{lemma}[Energy stability
 ]
 \label{lemma:energy-balance-continuous}
		The following differential energy balance equation holds
		\begin{equation}
  \label{eq:energy_continuous}
			\partial_t E(s_a) + \norm[1]{\sqrt{\lambda_a \kappa}\nabla p_a}^2_\Omega + \norm[1]{\sqrt{\lambda_\ell \kappa}\nabla p_\ell}^2_\Omega = (q_a, p_a)_\Omega + (q_\ell,p_\ell)_\Omega + (\lambda_a \kappa \nabla p_a \cdot \boldsymbol{n}, p_a)_{\partial \Omega} + \del{\lambda_\ell \kappa \nabla p_\ell \cdot \boldsymbol{n}, p_\ell}_{\partial \Omega},
		\end{equation}
		where $\displaystyle E(s_a) = \int_\Omega \phi F(s_a) \dif x$.
  \label{lem:cont_energy_balance}
	\end{lemma}
	
	\begin{proof}
		
Note that \eqref{eq:chain_rule} yields
		\begin{equation}
			\partial_t E(s_a) = \int_\Omega \phi \partial_t F(s_a) \dif x = \int_\Omega \phi \nu_a \partial_t s_a \dif x.
			\label{eq:energy_1}
		\end{equation}
		Testing \cref{eq:sa_eq} with $\nu_a$, we have
		\begin{equation*}
			\del{\phi\partial_t s_a, \nu_a}_\Omega + \del{\nabla \cdot \del{\lambda_a \kappa p_c' \nabla s_a}, \nu_a}_\Omega - \del{\nabla \cdot \del{\lambda_a \kappa  \nabla p_\ell}, \nu_a}_\Omega = (q_a, \nu_a)_\Omega,
		\end{equation*}
		which can be written as
		\begin{equation*}
			\del{\phi\partial_t s_a, \nu_a}_\Omega + \del{\nabla \cdot \del{\lambda_a \kappa  \nabla p_c}, \nu_a}_\Omega - \del{\nabla \cdot \del{\lambda_a \kappa  \nabla p_\ell}, \nu_a}_\Omega = (q_a, \nu_a)_\Omega.
		\end{equation*}		
		Recall that  by the definition of the capillary pressure \eqref{eq:algebraic_relations} we have $p_a = p_\ell - p_{c}$, so $\nabla p_a = \nabla p_\ell - \nabla p_c$.
 Thus,
		\begin{equation*}
			\del{\phi\partial_t s_a, \nu_a}_\Omega - \del{\nabla \cdot \del{\lambda_a \kappa \nabla p_a}, \nu_a}_\Omega  = (q_a, \nu_a)_\Omega,
		\end{equation*}
which by 
integration by parts 
yields
		\begin{equation*}
			\del{\phi\partial_t s_a, \nu_a}_\Omega + \del{\lambda_a \kappa \nabla p_a, \nabla \nu_a}_\Omega  = (q_a, \nu_a)_\Omega + (\lambda_a \kappa \nabla p_a \cdot \boldsymbol{n}, \nu_a)_{\partial \Omega}.
		\end{equation*}
Recalling that by \eqref{eq:chemicalpotential-capillarypressure} and \eqref{eq:algebraic_relations}  we have a relation between the chemical potential and the capillary pressure  $\nu_a = p_a - p_\ell$, 
hence
		\begin{multline}
			\del{\phi\partial_t s_a, \nu_a}_\Omega + \del{\lambda_a \kappa \nabla p_a, \nabla p_a}_\Omega  - \del{\lambda_a \kappa \nabla p_a, \nabla p_\ell}_\Omega = (q_a, p_a)_\Omega - (q_a, p_\ell)_\Omega
			\\
			+ (\lambda_a \kappa \nabla p_a \cdot \boldsymbol{n}, p_a)_{\partial \Omega} - (\lambda_a \kappa \nabla p_a \cdot \boldsymbol{n}, p_\ell)_{\partial \Omega}.
\label{eq:sa_aux_energy}
		\end{multline}
		On the other hand, we multiply \cref{eq:pl_eq} by $p_\ell$ and integrate over $\Omega$ to get
		\begin{equation*}
			-\del{\nabla\cdot\del{\lambda \kappa \nabla p_\ell}, p_\ell}_\Omega + \del{\nabla\cdot \del{\lambda_a\kappa \nabla p_c} ,p_\ell}_\Omega = \del{q,p_\ell}_\Omega.
		\end{equation*}
Recall again that \eqref{eq:algebraic_relations}  gives $p_a = p_\ell - p_{c}$, so $\nabla p_c = \nabla p_\ell - \nabla p_a$, and therefore
		\begin{equation*}
			-\del{\nabla\cdot\del{\lambda \kappa \nabla p_\ell}, p_\ell}_\Omega + \del{\nabla\cdot \del{\lambda_a\kappa\nabla p_\ell} ,p_\ell}_\Omega -  \del{\nabla\cdot \del{\lambda_a\kappa\nabla p_a} ,p_\ell}_\Omega = \del{q,p_\ell}_\Omega.
		\end{equation*}
		Since $\lambda = \lambda_\ell + \lambda_a$, we have
		\begin{equation*}
			-\del{\nabla\cdot\del{\lambda_\ell \kappa \nabla p_\ell}, p_\ell}_\Omega -  \del{\nabla\cdot \del{\lambda_a\kappa\nabla p_a} ,p_\ell}_\Omega = \del{q,p_\ell}_\Omega,
		\end{equation*}
thus integration by parts yields
		\begin{equation}
			\del{\lambda_\ell \kappa \nabla p_\ell, \nabla p_\ell}_\Omega  + \del{\lambda_a\kappa\nabla p_a ,\nabla p_\ell}_\Omega = \del{q,p_\ell}_\Omega + \del{\lambda_\ell \kappa \nabla p_\ell \cdot \boldsymbol{n}, p_\ell}_{\partial \Omega} + \del{\lambda_a \kappa \nabla p_a \cdot \boldsymbol{n}, p_\ell}_{\partial \Omega}.
\label{eq:pl_aux_energy}
		\end{equation}
Finally, adding up equations \eqref{eq:sa_aux_energy} and \eqref{eq:pl_aux_energy} 
gives
		\begin{multline*}
			\del{\phi\partial_t s_a, \nu_a}_\Omega + \del{\lambda_a \kappa \nabla p_a, \nabla p_a}_\Omega  + \del{\lambda_\ell \kappa \nabla p_\ell, \nabla p_\ell}_\Omega = (q_a, p_a)_\Omega + (q_\ell, p_\ell)_\Omega 
			\\
			+ (\lambda_a \kappa \nabla p_a \cdot \boldsymbol{n}, p_a)_{\partial \Omega} + \del{\lambda_\ell \kappa \nabla p_\ell \cdot \boldsymbol{n}, p_\ell}_{\partial \Omega},
		\end{multline*}
which together with \cref{eq:energy_1}
implies the energy balance equation
  \eqref{eq:energy_continuous}.
		%
	\end{proof}
	
Now that we have the relation \eqref{eq:energy_continuous} for the free energy in the continuous case, we turn our attention to the time discretization \eqref{eq:midpoint}. Note that the chain rule, hence  \eqref{eq:chain_rule}, is not valid at the 
semi-discrete in time level. 
In the following lemma, we first obtain an expression for the discrete time derivative of the free energy 
$
E(s_a) = \int_\Omega \phi F(s_a) \dif x$, corresponding to the chain rule relation \eqref{eq:chain_rule}.
\begin{lemma}
The discrete time derivative of the free energy satisfies
		\begin{equation}
			\frac{1}{\tau^n}\del[1]{E(s_a^{n+1}) - E(s_a^n)} = \frac{1}{\tau^n}\int_\Omega \phi \nu_a^{n+1/2}\del[1]{s_a^{n+1} - s_a^n}\dif x,
			\label{eq:discrete_time_energy}
		\end{equation}
		where
\begin{align}
\nu_a^{n+1/2} 
& 
= \gamma_a \frac{\ln(s_a^{n+1}) + \ln(s_a^{n})}{2} + \gamma_a s_a^{n+1/2}\frac{\ln(s_a^{n+1}) - \ln(s_a^{n})}{s_a^{n+1} - s_a^n} - \gamma_a
			- \gamma_\ell \frac{\ln(1-s_a^{n+1}) + \ln(1-s_a^{n})}{2} 
\label{eq:nu_discrete}
\\
& 
- \gamma_\ell \del[1]{1-s_a^{n+1/2}}\frac{\ln(1-s_a^{n+1}) - \ln(1-s_a^{n})}{s_a^{n+1} - s_a^n} + \gamma_\ell + \gamma_{a\ell}\del[1]{1 - 2s_a^{n+1/2}},
\notag
\end{align}
		is a second-order approximation of $\nu_a(s_a(t^{n+1/2}))$.
		\label{lem:discrete_energy_aux}
	\end{lemma}
	\begin{proof}
We aim to express the increment in the Gibbs free energy
$F(s_a^{n+1}) - F(s_a^n)$
		\begin{align*}
			F(s_a^{n+1}) - F(s_a^n)
			&= \gamma_a s_a^{n+1}\del[1]{\ln(s_a^{n+1}) - 1} + \gamma_\ell \del[1]{1 - s_a^{n+1}} \del[1]{\ln(1-s_a^{n+1}) - 1} + \gamma_{a\ell} s_a^{n+1}\del[1]{1-s_a^{n+1}}
			\\
			&\quad - \gamma_a s_a^{n}\del[1]{\ln(s_a^{n}) - 1} - \gamma_\ell \del[1]{1 - s_a^{n}} \del[1]{\ln(1-s_a^{n}) - 1} - \gamma_{a\ell} s_a^{n}\del[1]{1-s_a^{n}}
		\end{align*}
in terms of the increments of the aqueous saturation $s_a^{n+1} - s_a^n$.  

First, let us work on the terms $\gamma_a s_a^{n+1}\del[1]{\ln(s_a^{n+1}) - 1} - \gamma_a s_a^{n}\del[1]{\ln(s_a^{n}) - 1}$. We have:
\begin{align}
			s_a^{n+1}\del[1]{\ln(s_a^{n+1}) - 1} - s_a^{n}\del[1]{\ln(s_a^{n}) - 1} 
   & = \frac{1}{2}s_a^{n+1}\del[1]{\ln(s_a^{n}) - 1} + \frac{1}{2}s_a^{n+1}\del[1]{\ln(s_a^{n+1}) - \ln(s_a^{n})}
\label{eq:Fnplus1_Fn_difference}
\\
			& \quad + \frac{1}{2}s_a^{n}\del[1]{\ln(s_a^{n+1}) - 1} + \frac{1}{2}\del[1]{\ln(s_a^{n+1}) - 1}\del[1]{s_a^{n+1} - s_a^n}
			\nonumber
			\\
			& \quad - \frac{1}{2}s_a^{n}\del[1]{\ln(s_a^{n}) - 1} - \frac{1}{2}s_a^{n}\del[1]{\ln(s_a^{n}) - 1}
			\nonumber
			\\
			& = \frac{1}{2}\del[1]{s_a^{n+1} - s_a^n}\del[1]{\ln(s_a^{n+1}) + \ln(s_a^{n})} 
			\nonumber
			\\
			& \quad + \frac{1}{2}\del[1]{s_a^{n+1} + s_a^n}\del[1]{\ln(s_a^{n+1}) - \ln(s_a^{n})}  - \del[1]{s_a^{n+1} - s_a^n}.
\notag
\end{align}
Similarly,
		\begin{align*}
			\del[1]{1-s_a^{n+1}}\del[1]{\ln(1-s_a^{n+1}) - 1} - \del[1]{1-s_a^{n}}\del[1]{\ln(1-s_a^{n}) - 1} & = -\frac{1}{2}\del[1]{s_a^{n+1} - s_a^n}\del[1]{\ln(1-s_a^{n+1}) + \ln(1-s_a^{n})}
			\\
			& \quad - \frac{1}{2}\del[1]{s_a^{n+1} + s_a^n}\del[1]{\ln(1-s_a^{n+1}) - \ln(1-s_a^{n})} 
			\\
			& \quad + \del[1]{s_a^{n+1} - s_a^n} + \del[1]{\ln(1-s_a^{n+1}) - \ln(1-s_a^{n})}.
		\end{align*}
		Moreover,
		\begin{align*}
			s_a^{n+1}\del[1]{1-s_a^{n+1}} - s_a^{n}\del[1]{1-s_a^{n}} = \del[1]{s_a^{n+1} - s_a^n}\del[1]{1 - \del[0]{s_a^{n+1} + s_a^n}} = \del[1]{s_a^{n+1} - s_a^n}\del[1]{1 - 2s_a^{n+1/2}}.
		\end{align*}
		
		Thus, we have
		\begin{align*}
			F(s_a^{n+1}) - F(s_a^n)
			&= \del[1]{s_a^{n+1} - s_a^n} \left\{\gamma_a \frac{\ln(s_a^{n+1}) + \ln(s_a^{n})}{2} + \gamma_a s_a^{n+1/2}\frac{\ln(s_a^{n+1}) - \ln(s_a^{n})}{s_a^{n+1} - s_a^n} - \gamma_a\right.
			\\
			&\quad \left. - \gamma_\ell \frac{\ln(1-s_a^{n+1}) + \ln(1-s_a^{n})}{2} - \gamma_\ell \del[1]{1-s_a^{n+1/2}}\frac{\ln(1-s_a^{n+1}) - \ln(1-s_a^{n})}{s_a^{n+1} - s_a^n} + \gamma_\ell\right.
			\\
			& \quad \left. + \gamma_{a\ell}\del[1]{1 - 2s_a^{n+1/2}} \right\}
			\\
			&\equiv \del[1]{s_a^{n+1} - s_a^n} \nu_a^{n+1/2},
		\end{align*}
		where $\nu_a^{n+1/2}$ is given by \eqref{eq:nu_discrete}. 
  Then \eqref{eq:discrete_time_energy}, the discrete analog of the chain rule \eqref{eq:energy_1}, follows after recalling the definition of the free energy $E(s_a^n) = \int_\Omega\phi F(s_a^n)\dif x$. 
  
  Next, we show that $\nu_a^{n+1/2}$ given by \eqref{eq:nu_discrete} is a second-order approximation of $\nu_a(s_a(t^{n+1/2}))  = F'(s_a(t^{n+1/2}))$ (from  \eqref{eq:chain_rule} evaluated at $t^{n+1/2}$). 
  Recall that by \eqref{eq:free_energy} 
  and \eqref{eq:saturations_relation}
  the free Gibbs energy is $F(s_a) = \gamma_a s_a\del[1]{\ln(s_a) - 1} + \gamma_\ell \del[1]{1 - s_a} \del[1]{\ln(1-s_a) - 1} + \gamma_{a\ell} s_a\del[1]{1-s_a}$. Let us denote by $F_a(s_a)$ the terms that contain $\gamma_a$, i.e., $F_a(s_a) = s_a\del[1]{\ln(s_a) - 1}$. 
  Then, 
  a Taylor series expansion about $s_a(t^{n+1/2})$ 
  yields
		\begin{multline*}
			F_a(s_a^{n+1}) = F_a(s_a(t^{n+1/2})) + \del{s_a^{n+1} - s_a(t^{n+1/2})}F_a'(s_a(t^{n+1/2})) + \frac{1}{2}\del{s_a^{n+1} - s_a(t^{n+1/2})}^2F_a''(s_a(t^{n+1/2})) 
			\\
			+ \mathcal{O}\del{(s_a^{n+1} - s_a(t^{n+1/2}))^3},
		\end{multline*}
and similarly,
		\begin{multline*}
			F_a(s_a^{n}) = F_a(s_a(t^{n+1/2})) + \del{s_a^{n} - s_a(t^{n+1/2})}F_a'(s_a(t^{n+1/2})) + \frac{1}{2}\del{s_a^{n} - s_a(t^{n+1/2})}^2F_a''(s_a(t^{n+1/2})) 
			\\
			+ \mathcal{O}\del{(s_a^{n} - s_a(t^{n+1/2}))^3}.
		\end{multline*}
Assuming $s_a$ is smooth enough in time, we have $(s_a^{n+1} - s_a(t^{n+1/2}))^3 = \mathcal{O}((\tau^n)^3)$ and $(s_a^{n} - s_a(t^{n+1/2}))^3 = \mathcal{O}((\tau^n)^3)$, 
and therefore
		\begin{multline*}
			(s_a^{n+1} - s_a^n)F_a'(s_a(t^{n+1/2})) = F_a(s_a^{n+1}) - F_a(s_a^n)
\\
			+ \frac{1}{2}\sbr{\del{s_a^{n+1} - s_a(t^{n+1/2})}^2 - \del{s_a^{n} - s_a(t^{n+1/2})}^2}F_a''(s_a(t^{n+1/2})) + \mathcal{O}((\tau^n)^3).
		\end{multline*}
Note that
		\begin{equation*}
			\del{s_a^{n+1} - s_a(t^{n+1/2})}^2 - \del{s_a^{n} - s_a(t^{n+1/2})}^2 = (s_a^{n+1} - s_a^n)(s_a^{n+1} - 2s_a(t^{n+1/2}) + s_a^n),
		\end{equation*}
where, also by Taylor expansion, 
$(s_a^{n+1} - 2s_a(t^{n+1/2}) + s_a^n)$ is of order $(\tau^n)^2$. 
Thus
		$$
		F_a'(s_a(t^{n+1/2})) = \frac{F_a(s_a^{n+1}) - F_a(s_a^n)}{s_a^{n+1} - s_a^n}  + C(\tau^n)^2.
		$$
		As in the expression \eqref{eq:Fnplus1_Fn_difference} above, we 
  obtain
		$$
		\abs{F_a'(s_a(t^{n+1/2})) -  \del{\frac{\ln(s_a^{n+1}) + \ln(s_a^{n})}{2} +  s_a^{n+1/2}\frac{\ln(s_a^{n+1}) - \ln(s_a^{n})}{s_a^{n+1} - s_a^n} -1}} \leq C(\tau^n)^2.
		$$
		Proceeding in a similar manner with the other terms of 
  $\nu_a(s_a^{n+1/2})$ 
  yields
		$$
		\abs{\nu_a(s_a(t^{n+1/2})) - \nu_a^{n+1/2}} \leq C(\tau^n)^2,
		$$
which concludes the argument.		
	\end{proof}
	
	\begin{remark}
		We note that Lemma \ref{lem:discrete_energy_aux} is independent of the choice of the time-stepping scheme. However, the expression \eqref{eq:nu_discrete} appears naturally from the discrete time derivative of the free energy. 
  Moreover, \eqref{eq:nu_discrete}  
  gives a second-order approximation of the chemical potential $\nu_a(s_a(t^{n+1/2}))$, which makes the midpoint method (i.e., taking $\theta^n=1/2$) a natural choice to obtain 
  a discrete version of the energy balance relation \eqref{eq:energy_continuous},
  with zero numerical dissipation.
	\end{remark}
	
	Next, we show that solution of the midpoint 
 method \eqref{eq:midpoint} satisfies a discrete-time energy balance equation, similar to the one obtained at the continuous-time level in Lemma \ref{lemma:energy-balance-continuous}.
 We emphasize here the essential role played by the relation between the chemical potential and the capillary pressure \eqref{eq:chemicalpotential-capillarypressure}. Since $\nu_a = -p_c$, then \cref{eq:nu_discrete} also gives us a second-order approximation to $p_c(s_a(t^{n+1/2}))$. This  enforces at discrete-time level $t^{n+1/2}$ the equality between  $p_c^{n+1/2}$ (computed with the midpoint method) and the chemical potential $\nu_a^{n+1/2}$.
	
	\begin{lemma}
 \label{lemma:energy-balance-discrete}
		The time discretization \cref{eq:system_iterations,eq:forward_step}, equivalently \eqref{eq:midpoint}, with $\theta^n = 1/2$ satisfy the following discrete energy balance equation
		%
\begin{align}
\label{eq:energy_discrete}
& 
\frac{1}{\tau^n}\del[1]{E(s_a^{n+1}) - E(s_a^n)} + \norm[1]{\sqrt{\lambda_\ell^{n+1/2} \kappa} \nabla p_\ell^{n+1/2}}_\Omega^2 + \norm[1]{\sqrt{\lambda_{a}^{n+1/2}\kappa} \nabla p_a^{n+1/2}}_\Omega^2  
= \del[1]{q_\ell^{n+1/2}, p_\ell^{n+1/2}}_\Omega \\
&
			+ \del[1]{q_a^{n+1/2}, p_a^{n+1/2}}_\Omega 
			+ \del{\lambda_\ell^{n+1/2} \kappa \nabla p_\ell^{n+1/2} \cdot \boldsymbol{n}, p_\ell^{n+1/2}}_{\partial \Omega} + \del{\lambda_a^{n+1/2} \kappa \nabla p_a^{n+1/2} \cdot \boldsymbol{n}, p_a^{n+1/2}}_{\partial \Omega}.
   \notag
\end{align}
		%
		\label{lem:discrete_energy_stab}
 \end{lemma}
	\begin{proof}
		Test \cref{eq:BE_sa} and \cref{eq:FE_sa} with $(1/2) \nu_a^{n+1/2}$, where $\nu_a^{n+1/2}$ is given by \eqref{eq:nu_discrete}, and add up to obtain
		\begin{multline*}
			\frac{1}{\tau^n}\del{\phi(s_a^{n+1} - s_a^n), \nu_a^{n+1/2}}_\Omega + \del{\nabla \cdot \del{\lambda_{a}^{n+1/2}\kappa \nabla p_c^{n+1/2}}, \nu_a^{n+1/2}}_\Omega 
	\\
			- \del{\nabla\cdot\del{\lambda_a^{n+1/2}\kappa\nabla p_\ell^{n+1/2}}, \nu_a^{n+1/2}}_\Omega = \del[1]{q_a^{n+1/2},\nu_a^{n+1/2}}_\Omega.
		\end{multline*}
Similarly to
  the continuous case, we 
  write this as
		\begin{equation*}
			\frac{1}{\tau^n}\del{\phi(s_a^{n+1} - s_a^n), \nu_a^{n+1/2}}_\Omega - \del{\nabla \cdot \del{\lambda_{a}^{n+1/2}\kappa \nabla p_a^{n+1/2}}, \nu_a^{n+1/2}}_\Omega = \del[1]{q_a^{n+1/2},\nu_a^{n+1/2}}_\Omega,
		\end{equation*}
and apply integration by parts to get
		\begin{multline*}
			\frac{1}{\tau^n}\del{\phi(s_a^{n+1} - s_a^n), \nu_a^{n+1/2}}_\Omega + \del{\lambda_{a}^{n+1/2}\kappa \nabla p_a^{n+1/2}, \nabla \nu_a^{n+1/2}}_\Omega = \del[1]{q_a^{n+1/2},\nu_a^{n+1/2}}_\Omega 
\\
			+ \del{\lambda_{a}^{n+1/2}\kappa \nabla p_a^{n+1/2}\cdot \boldsymbol{n}, \nu_a^{n+1/2}}_{\partial \Omega}.
		\end{multline*}
  Recall that, at the continuous level, we have $\nu_a = p_a - p_{\ell}$. At the discrete level, we have $\nu_a^{n+1/2}$ defined by \cref{eq:nu_discrete}, and $p_\ell^{n+1/2}$ which comes from the solution of the discrete system \cref{eq:system_midstep}. Therefore, to enforce this condition, we set $p_a^{n+1/2} = \nu_a^{n+1/2}+p_\ell^{n+1/2}$. This yields
		%
		\begin{multline}
			\frac{1}{\tau^n}\del{\phi(s_a^{n+1} - s_a^n), \nu_a^{n+1/2}}_\Omega + \del{\lambda_{a}^{n+1/2}\kappa \nabla p_a^{n+1/2}, \nabla p_a^{n+1/2}}_\Omega - \del{\lambda_{a}^{n+1/2}\kappa \nabla p_a^{n+1/2}, \nabla p_\ell^{n+1/2}}_\Omega 
			\\
			= \del[1]{q_a^{n+1/2},\nu_a^{n+1/2}}_\Omega
			+ \del{\lambda_{a}^{n+1/2}\kappa \nabla p_a^{n+1/2}\cdot \boldsymbol{n}, \nu_a^{n+1/2}}_{\partial \Omega}.
\label{eq:sa_aux_energy_disc}
		\end{multline}
  
On the other hand, test \cref{eq:BE_pl} with $p_\ell^{n+1/2}$, and use 
$\nabla p_c = \nabla p_\ell - \nabla p_a$, 
$\lambda = \lambda_a + \lambda_\ell$,
to obtain 
		\begin{equation*}
			-\del{\nabla \cdot \left(\lambda_\ell^{n+1/2} \kappa \nabla p_\ell^{n+1/2}\right), p_\ell^{n+1/2} } _\Omega
			- \del{\nabla \cdot \left(\lambda_a^{n+1/2} \kappa \nabla p_{a}^{n+1/2}\right), p_\ell^{n+1/2}}_\Omega = \del[1]{q^{n+1/2}, p_\ell^{n+1/2}}_\Omega,
		\end{equation*}
which, integrating 
by parts, gives
		\begin{align}
& 
\del{\lambda_\ell^{n+1/2} \kappa \nabla p_\ell^{n+1/2}, \nabla p_\ell^{n+1/2} } _\Omega
			+ \del{\lambda_a^{n+1/2} \kappa \nabla p_{a}^{n+1/2}, \nabla p_\ell^{n+1/2}}_\Omega 
\label{eq:pl_aux_energy_disc}
\\
   & 
   = \del[1]{q^{n+1/2}, p_\ell^{n+1/2}} _\Omega
			+ \del{\lambda_\ell^{n+1/2} \kappa \nabla p_\ell^{n+1/2} \cdot \boldsymbol{n}, p_\ell^{n+1/2}}_{\partial \Omega} + \del{\lambda_a^{n+1/2} \kappa \nabla p_a^{n+1/2} \cdot \boldsymbol{n}, p_\ell^{n+1/2}}_{\partial \Omega}.
\notag
\end{align}
Finally, adding up \eqref{eq:sa_aux_energy_disc} and \eqref{eq:pl_aux_energy_disc}, and using \eqref{eq:discrete_time_energy}, 
yields the discrete energy balance \cref{eq:energy_discrete}.
	\end{proof}

 \begin{remark}
     We note that there are no numerical dissipation terms in \eqref{lem:discrete_energy_stab}.
     The result \eqref{eq:energy_discrete} in Lemma \ref{lem:discrete_energy_stab} is the discrete-time equivalent of the result \eqref{eq:energy_continuous} in Lemma \ref{lem:cont_energy_balance} for the continuous case. 
This conservation of the structural discrete-time energy balance, 
respectively of the discrete-time Dirac structure, 
represents an extension of the notion of symplecticity of geometric integration schemes to open systems \cite{MR4001128,MR4586824,10.5555/553988}.
 \end{remark}
	\begin{remark}
Note that if there are no forcing terms, i.e., $q_j= 0$, and $\lambda_j^{n+1/2} \kappa \nabla p_j^{n+1/2} \cdot \boldsymbol{n}|_{\partial\Omega} = 0$, for $j=a,\ell$, then Lemma \ref{lem:discrete_energy_stab} implies the stability of the method \eqref{eq:midpoint}, in the sense that the free energy dissipates:
		\begin{equation*}
			E(s_a^{n+1}) + \tau^n \norm[1]{\sqrt{\lambda_\ell^{n+1/2} \kappa} \nabla p_\ell^{n+1/2}}_\Omega^2 + \tau^n \norm[1]{\sqrt{\lambda_{a}^{n+1/2}\kappa} \nabla p_a^{n+1/2}}_\Omega^2 = E(s_a^n).
		\end{equation*}
	\end{remark}
	
\section{Numerical results}
\label{sec:numerical_results}
In this section, we 
test the performance of the numerical method in regards to rates of convergence, long-time discretization errors, and free energy errors. 
For the spatial discretization of the problem we use an interior penalty discontinuous Galerkin method with piecewise polynomials of order $k$, see e.g., ~\cite{sosajones}. In~\cite{sosajones} it is proven that the spatial discretization error is of order $k+1$. 

For the 
first two tests, we use the method of manufactured solutions, i.e., the right hand sides (sinks/sources) and boundary conditions are calculated so that the exact solution of the problem \eqref{eq:PDE_system} is
\begin{equation}
    p_\ell = \frac{e^t}{e^T}(2 + xy^2 + x^2\sin(y)), \quad s_a = \frac{e^t(2 + x^2y^2 + \cos(x))}{8 e^T},
\label{eq:analytical_sol}
\end{equation}
where $T$ is the final time of the simulation. The rest of the parameters are given by
\begin{align*}
    &\kappa_\ell = s_\ell(s_\ell + s_a)(1-s_a), \quad \kappa_a = s_a^2,\quad \kappa = 1, \quad \phi=0.2,
    \\
    &p_{c}(s_a) = \frac{6.3}{\ln(0.01)}\ln(s_a), \quad \mu_\ell = 0.75, \quad \mu_a = 0.5.
\end{align*}
%
The iterations in the backward Euler step \eqref{eq:BE-iteration} at time level $t^{n+\theta^n}$ are stopped when 
\begin{equation}
        \max\left\{\frac{\norm[1]{p^{n+\theta^n}_{\ell_{(i+1)}} - p^{n+\theta^n}_{\ell_{(i)}}}_\Omega}{\norm[1]{p^{n+\theta^n}_{\ell_{(i)}}}_\Omega}, \frac{\norm[1]{s^{n+\theta^n}_{a_{(i+1)}} - s^{n+\theta^n}_{a_{(i)}}}_\Omega}{\norm[1]{s^{n+\theta^n}_{a_{(i)}}}_\Omega}\right\} < \text{TOL},
        \label{eq:stop_crit}
\end{equation}
where $\text{TOL}$ is a user-given parameter.

We will compare four methods: the midpoint method, i.e., $\theta^n = 1/2$ in \eqref{eq:midpoint}, which we refer to as \textbf{MP}; the backward Euler method ($\theta^n=1$ in \eqref{eq:midpoint}), referred to as \textbf{BE}; and a first and second-order time-lagging schemes, which we denote by 
\ref{eq:first_lag} and 
\ref{eq:second_lag}, respectively. 
These time-lagging schemes do not require a subiteration, but instead, the coefficients that depend on the primary unknowns at time level $t^{n+1}$ are evaluated using first and second-order extrapolations of the primary unknowns. 
The first-order time-lagging method is given by
    \begin{align}
\label{eq:first_lag}
\tag{TL1}    
\begin{array}{l}
        -\nabla \cdot \left(\lambda^{n} \kappa \nabla p_{\ell}^{n+1}\right) 
        = q^{n+1} - \nabla \cdot \left(\lambda_{a}^{n} \kappa \nabla p_{c}^{n}\right),
\displaystyle
        \\
\displaystyle        
        \phi \frac{s_{a}^{n+1} - s_a^n}{\tau} + \nabla \cdot \left(\lambda_{a}^{n} \kappa p'^{n}_{c} \nabla s_{a}^{n+1}\right)  
        = q_a^{n+1} + \nabla \cdot \left(\lambda_{a}^{n} \kappa \nabla p_{\ell}^{n+1}\right).
\end{array}
    \end{align}
The second-order time-lagging scheme is based on
    \begin{align}
\label{eq:second_lag}
\tag{TL2}
\begin{array}{ll}
\displaystyle
-\nabla \cdot \left(\del[1]{2\lambda^{n} -\lambda^{n-1}}\kappa \nabla p_{\ell}^{n+1}\right) 
= q^{n+1} - \nabla \cdot \left(\del[1]{2\lambda_{a}^{n} -\lambda_a^{n-1}}\kappa \nabla \del[1]{2p_{c}^{n} - p_c^{n-1}} \right),
        \\
\displaystyle        
        \phi \frac{3s_{a}^{n+1} - 4s_a^n + s_a^{n-1}}{2\tau} + \nabla \cdot \left(\del[1]{2\lambda_{a}^{n} -\lambda_a^{n-1}} \kappa \del[1]{2p_c^{'n} -p_c'^{n-1}} \nabla s_{a}^{n+1}\right)
= q_a^{n+1} + \nabla \cdot \left(\del[1]{2\lambda_{a}^{n} -\lambda_a^{n-1}} \kappa \nabla p_{\ell}^{n+1}\right).
\end{array}
    \end{align}

\subsection{Rates of convergence}
\label{subsec:rates}

First, we 
evaluate the rates of convergence of the method presented in this work using $\theta^n = 1/2$ and $\theta^n = 1$, and we compare them 
with the rates of convergence of the time-lagging schemes \eqref{eq:first_lag} and \eqref{eq:second_lag} described above. 
For the spatial discretization, we use piecewise linear polynomials using a quadrilateral mesh of size $h$. The constant time step $\tau^n$ is taken equal to the mesh size $h$. 
The tolerance in \cref{eq:stop_crit} is taken to be $\text{TOL}=10^{-5}$.  The $L^2$-norm of the error at the final time $T=1$ and the rates of convergence are shown in \Cref{tab:ConvRates_tau_equals_h_MP,tab:ConvRates_tau_equals_h_BE,tab:ConvRates_tau_equals_h_timelag,tab:ConvRates_tau_equals_h_timelag2}.
\begin{center}
    \begin{table}[tbp]
        \centering
        \small
        \caption{Rates of convergence for \textbf{MP} with $\tau = h$, see \Cref{subsec:rates}.}
        \begin{tabular}{lrcccc}
            \toprule
            \multirow{2}{*}{}
            & & \multicolumn{2}{c}{$p_\ell$}  & \multicolumn{2}{c}{$s_a$} \\
            $\tau$ & {DOFs}
            & {$L^2(\Omega)$-error} & {Rate}
            & {$L^2(\Omega)$-error} & {Rate} \\
            \midrule
            {0.5} & {16} & {3.94e-2} & {-} & {1.72e-2} & {-} \\
            {0.25} & {64} & {1.16e-2} & {1.76} & {3.64e-3} & {2.24} \\
            {0.125} & {256}  & {3.16e-3} & {1.88} & {8.04e-4} & {2.18} \\
            {0.0625} & {1,024} & {8.26e-4} & {1.94} & {1.93e-4} & {2.06} \\
            {0.03125} & {4,096} & {2.12e-4} & {1.96} & {4.74e-5} & {2.03} \\
            {0.015625} & {16,384} & {5.38e-5} & {1.98} & {1.17e-5} & {2.02} \\
            {0.0078125} & {65,536} & {1.35e-5} & {1.99} & {2.93e-6} & {2.00} \\
            \bottomrule
        \end{tabular}
        \label{tab:ConvRates_tau_equals_h_MP}
    \end{table}
\end{center}
\begin{center}
\begin{table}[tbp]
    \centering
    \small
    \caption{Rates of convergence for \textbf{BE} with $\tau = h$, see \Cref{subsec:rates}.}
    \begin{tabular}{lrcccc}
        \toprule
        \multirow{2}{*}{}
        & & \multicolumn{2}{c}{$p_\ell$}  & \multicolumn{2}{c}{$s_a$} \\
        $\tau$ & {DOFs}
        & {$L^2(\Omega)$-error} & {Rate}
        & {$L^2(\Omega)$-error} & {Rate} \\
        \midrule
        {0.5} & {16} & {5.21e-2} & {-} & {2.34e-3} & {-} \\
        {0.25} & {64} & {1.37e-2} & {1.93} & {7.74e-4} & {1.60} \\
        {0.125} & {256}  & {3.66e-3} & {1.90} & {3.19e-4} & {1.28} \\
        {0.0625} & {1,024} & {1.00e-3} & {1.87} & {1.53e-4} & {1.06} \\
        {0.03125} & {4,096} & {2.93e-4} & {1.77} & {7.64e-5} & {1.00} \\
        {0.015625} & {16,384} & {9.49e-5} & {1.63} & {3.84e-5} & {0.99} \\
        {0.0078125} & {65,536} & {3.50e-5} & {1.44} & {1.93e-5} & {0.99} \\
        \bottomrule
    \end{tabular}
    \label{tab:ConvRates_tau_equals_h_BE}
\end{table}
\end{center}
\begin{center}
\begin{table}[tbp]
    \centering
    \small
    \caption{Rates of convergence for 
    \ref{eq:first_lag} 
    with $\tau = h$, see \Cref{subsec:rates}.}
    \begin{tabular}{lrcccc}
        \toprule
        \multirow{2}{*}{}
        & & \multicolumn{2}{c}{$p_\ell$}  & \multicolumn{2}{c}{$s_a$} \\
        $\tau$ & {DOFs}
        & {$L^2(\Omega)$-error} & {Rate}
        & {$L^2(\Omega)$-error} & {Rate} \\
        \midrule
        {0.5} & {16} & {4.37e-2} & {-} & {2.91e-2} & {-} \\
        {0.25} & {64} & {8.19e-3} & {2.41} & {1.60e-2} & {0.86} \\
        {0.125} & {256}  & {4.73e-3} & {0.79} & {8.33e-3} & {0.94} \\
        {0.0625} & {1,024} & {3.67e-3} & {0.37} & {4.38e-3} & {0.92} \\
        {0.03125} & {4,096} & {2.28e-3} & {0.69} & {2.28e-3} & {0.94} \\
        {0.015625} & {16,384} & {1.27e-3} & {0.84} & {1.16e-3} & {0.97} \\
        {0.0078125} & {65,536} & {6.74e-4} & {0.91} & {5.92e-4} & {0.97} \\
        \bottomrule
    \end{tabular}    \label{tab:ConvRates_tau_equals_h_timelag}
\end{table}
\end{center}
\begin{center}
\begin{table}[tbp]
    \centering
    \small
    \caption{Rates of convergence for 
    \ref{eq:second_lag} with $\tau = h$, see \Cref{subsec:rates}.}
    \begin{tabular}{lrcccc}
        \toprule
        \multirow{2}{*}{}
        & & \multicolumn{2}{c}{$p_\ell$}  & \multicolumn{2}{c}{$s_a$} \\
        $\tau$ & {DOFs}
        & {$L^2(\Omega)$-error} & {Rate}
        & {$L^2(\Omega)$-error} & {Rate} \\
        \midrule
        {0.5} & {16} & {5.12e-2} & {-} & {1.01e-2} & {-} \\
        {0.25} & {64} & {1.05e-2} & {2.28} & {3.95e-3} & {1.35} \\
        {0.125} & {256}  & {2.58e-3} & {2.02} & {1.08e-3} & {1.87} \\
        {0.0625} & {1,024} & {6.39e-4} & {2.01} & {2.90e-4} & {1.89} \\
        {0.03125} & {4,096} & {1.60e-4} & {1.99} & {7.52e-5} & {1.94} \\
        {0.015625} & {16,384} & {4.02e-5} & {1.99} & {1.91e-5} & {1.97} \\
        {0.0078125} & {65,536} & {1.07e-5} & {1.91} & {4.82e-6} & {1.98} \\
        \bottomrule
    \end{tabular}
    \label{tab:ConvRates_tau_equals_h_timelag2}
\end{table}
\end{center}
As expected, we see that the backward Euler method and the first-order time lagging method both converge with order one. The midpoint method is second-order. Moreover, we observe that as $\tau \rightarrow 0$, the errors with \textbf{MP} are lower than with \textbf{BE}. 

 \subsection{Long-time behavior}
\label{subsec:longtime}
We note that in the context of simulation of multiphase flow in porous media, usually the final time of the simulation is quite large, as undergound flow tends to occur over long periods of time. 
Therefore, it is important for the time-stepping method to remain accurate for large final times. 
We consider the same solution given in \cref{eq:analytical_sol}, with the final time $T=20$. 
Note that $p_\ell$ and $s_a$ in \eqref{eq:analytical_sol} grow exponentially in time, but always remain bounded due to the factor $e^T$.  We consider two different time steps: $\tau^n = 0.05$, and $\tau^n = 1$, for all $n$, and piecewise linear polynomials on a mesh that contains $64\times 64$ quadrilaterals. 

In Figure \ref{fig:longtime_errors} we compare the $L^2$ errors in the liquid pressure and aqueous saturation at each time level, for the \textbf{MP}, \textbf{BE}, 
\ref{eq:first_lag} and 
\ref{eq:second_lag} methods, for both values of the time steps $\tau^n = 0.05$, and $\tau^n = 1$. 
We see that at all $t^{n}$ time values, the errors in the midpoint method values are significantly lower than with the backward Euler method, for $\tau^n = 0.05$. For $\tau^n = 1$, the $L^2$ errors in $s_a$ obtained with \textbf{MP} are not the lowest ones anymore. Since the discretization error for backward Euler is of order $\tau$, whereas for the midpoint it is of order $\tau^2$, it is expected that for $\tau \geq 1$, \textbf{BE} produces lower errors than \textbf{MP}. However, we can observe that \textbf{MP} still produces errors that are comparable to those obtained with \textbf{BE}.
In Table \ref{tab:maxerrors} we report the maximum $L^2$ error over time. 
(i) We observe that for $\tau^n=0.05$, the midpoint method has an error one order of magnitude lower than the backward Euler method for the liquid pressure. (ii) Furthermore, for $\tau^n=1$, the error in $p_\ell$ with the midpoint method is still one order of magnitude lower than the error with the backward Euler method.
\begin{figure}
    \begin{subfigure}[t]{.45\textwidth}
        \centering
        \includegraphics[width=\linewidth]{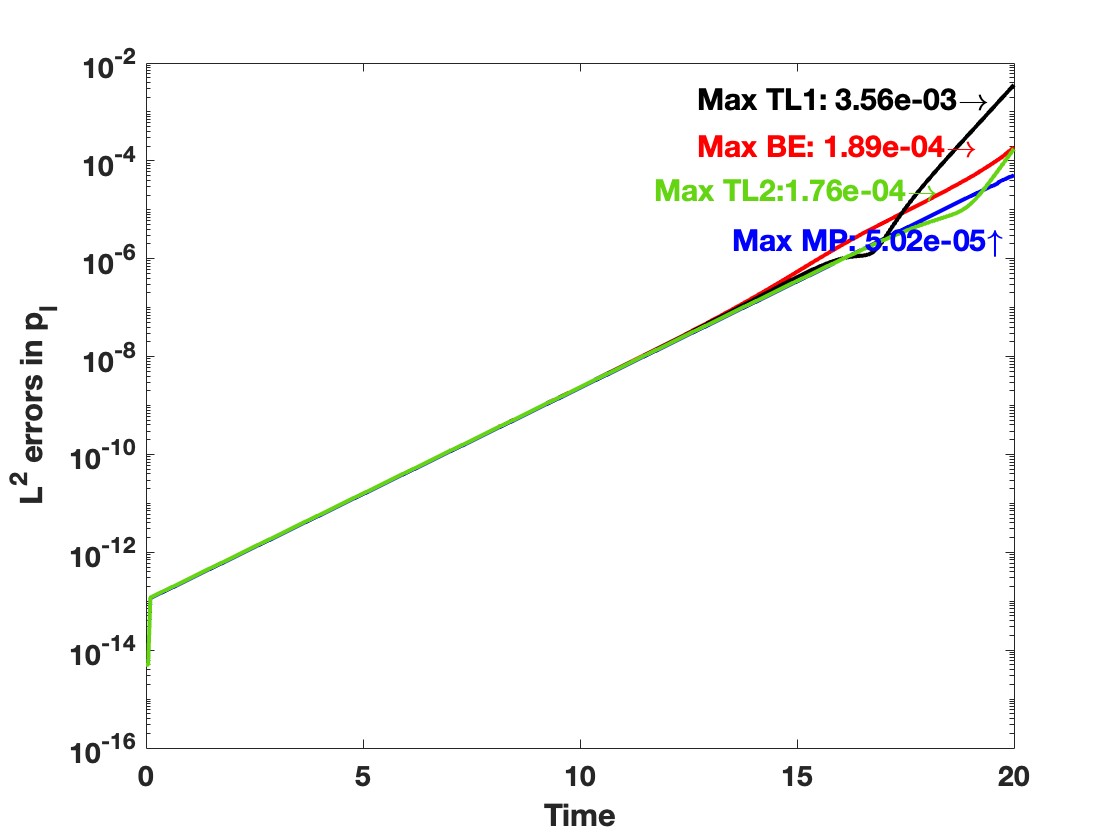}
        \caption{Liquid pressure, $\tau^n = 0.05$.}
        \label{subfig:longtime_pl_0p1}
    \end{subfigure}
    \begin{subfigure}[t]{.45\textwidth}
        \centering
        \includegraphics[width=\linewidth]{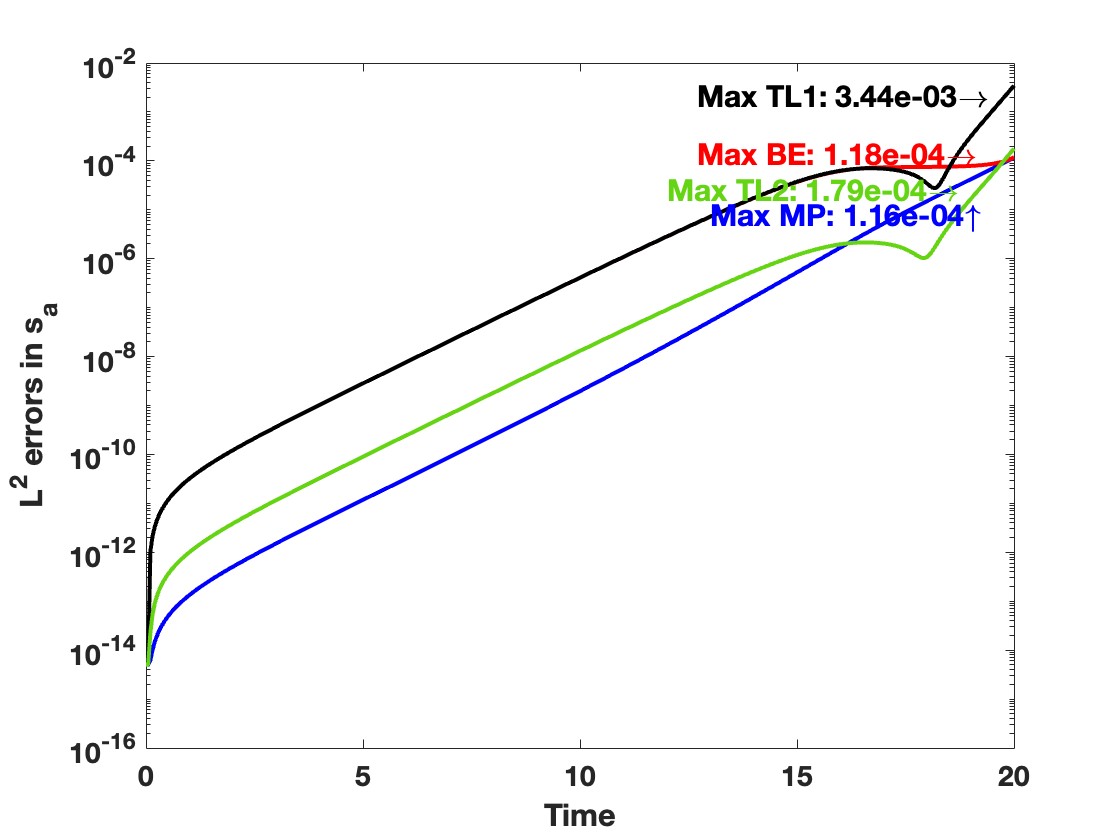}
        \caption{Aqueous saturation, $\tau^n = 0.05$.}
        \label{subfig:longtime_sa_0p1}
    \end{subfigure}
\\
        \begin{subfigure}[t]{.45\textwidth}
    \centering
    \includegraphics[width=\linewidth]{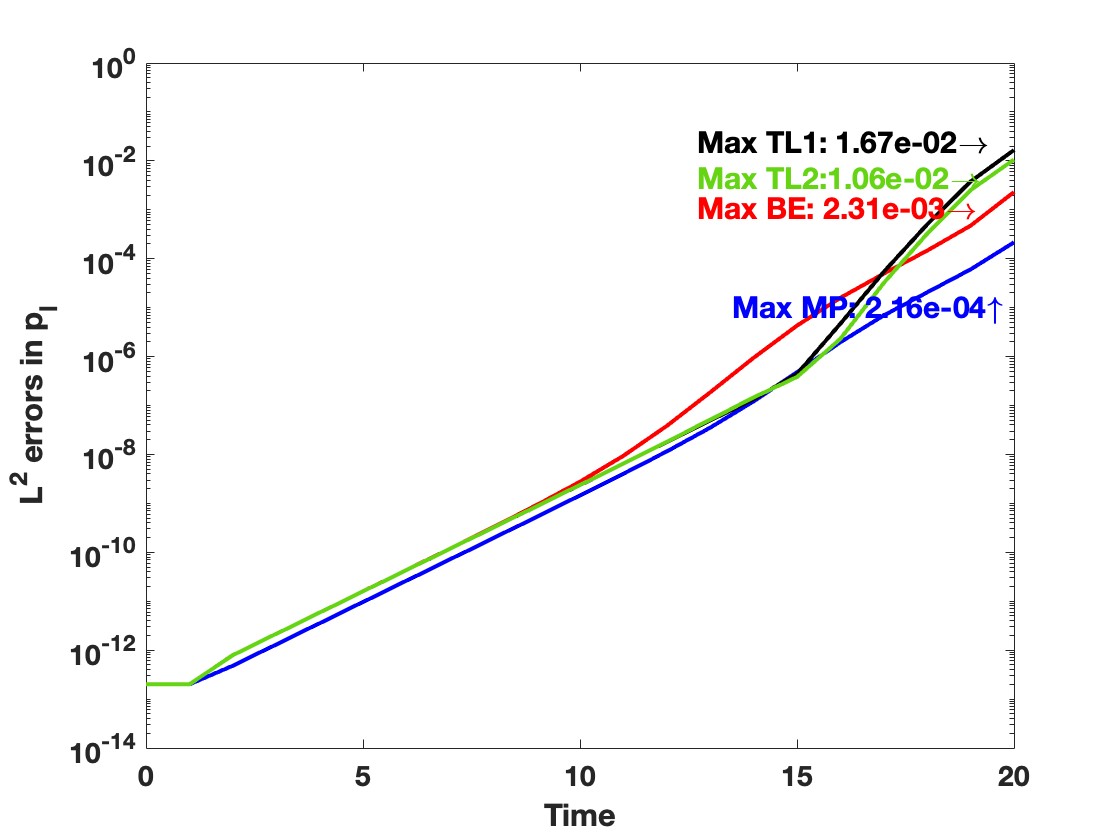}
    \caption{Liquid pressure, $\tau^n = 1$.}
    \label{subfig:longtime_pl_2}
\end{subfigure}
\begin{subfigure}[t]{.45\textwidth}
    \centering
    \includegraphics[width=\linewidth]{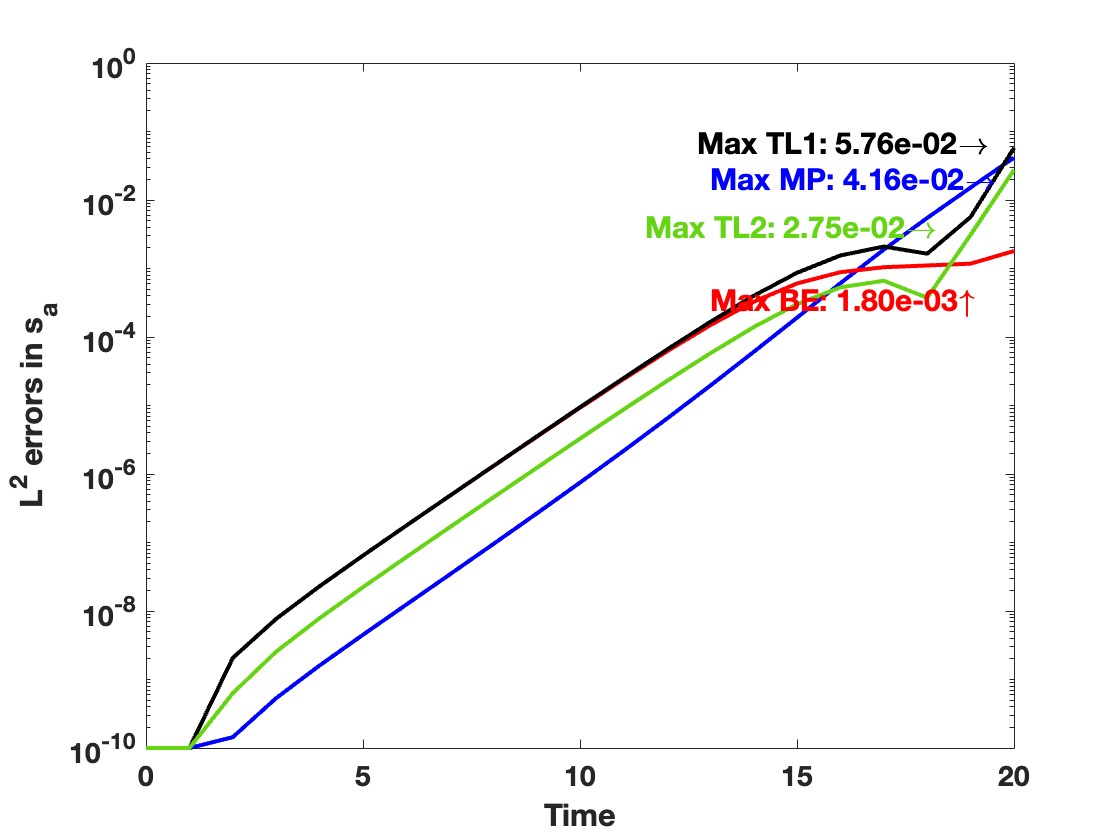}
    \caption{Aqueous saturation, $\tau^n = 1$.}
    \label{subfig:longtime_sa_2}
\end{subfigure}
    \caption{$L^2$ norm of the error in $p_\ell$ (left) and $s_a$ (right) for example in \cref{subsec:longtime}. Top row: $\tau^n = 0.05$ and bottom row: $\tau^n = 1$. Blue line: \textbf{MP}, red line: \textbf{BE}, black line: 
    \ref{eq:first_lag}, green line: 
    \ref{eq:second_lag}.}
    \label{fig:longtime_errors}
\end{figure}
\begin{center}
    \begin{table}[tbp]
        \centering
        \small
        \caption{Maximum values over time of the $L^2$ error in $p_\ell$ and $s_a$ for test case in \Cref{subsec:longtime}.}
        \begin{tabular}{ccc|cc}
            \toprule
            & \multicolumn{2}{c}{$\tau^n = 0.05$} & \multicolumn{2}{c}{$\tau^n = 1$}
            \\
            \midrule
            Method
            & $\max_{n}\norm[1]{p_\ell^n - p_\ell(t^n)}_\Omega$  & $\max_{n}\norm[1]{s_a^n - s_a(t^n)}_\Omega$ & $\max_{n}\norm[1]{p_\ell^n - p_\ell(t^n)}_\Omega$  & $\max_{n}\norm[1]{s_a^n - s_a(t^n)}_\Omega$
            \\
            \midrule
            \textbf{MP} & $5.02\times 10^{-5}$ & $1.16\times 10^{-4}$ & $2.16\times 10^{-4}$ & $4.16\times 10^{-2}$
            \\
            \textbf{BE} & $1.89\times 10^{-4}$ & $1.18\times 10^{-4}$ & $2.31\times 10^{-3}$ & $1.80\times 10^{-3}$
            \\
            \ref{eq:first_lag} & $3.56\times 10^{-3}$ & $3.44\times 10^{-3}$ & $1.67\times 10^{-2}$ & $5.76\times 10^{-2}$
            \\
            \ref{eq:second_lag}& $1.76\times 10^{-4}$ & $1.79\times 10^{-4}$ & $1.06\times 10^{-2}$ & $2.75\times 10^{-2}$
            \\
            \bottomrule
        \end{tabular}
        \label{tab:maxerrors}
    \end{table}
\end{center}

\begin{center}
    \begin{table}[tbp]
        \centering
        \small
        \caption{Maximum values of the relative error in the free energy for test case in \Cref{subsec:longtime}.}
        \begin{tabular}{cc|c}
            \toprule
            & \multicolumn{1}{c}{$\tau^n = 0.05$} & $\tau^n = 1$
            \\
            \midrule
            Method
            & $\displaystyle \max_{n}\abs{\frac{E(s_a(t^n)) - E(s_a^n)}{E(s_a(t^n))}}$  & $\displaystyle \max_{n}\abs{\frac{E(s_a(t^n)) - E(s_a^n)}{E(s_a(t^n))}}$
            \\
            \midrule
            \textbf{MP} & $2.31\times 10^{-4}$ & $8.35\times 10^{-2}$
            \\
            \textbf{BE} & $2.35\times 10^{-2}$ & $5.30\times 10^{-1}$
            \\
            \ref{eq:first_lag} & $2.35\times 10^{-2}$ & $5.31\times 10^{-1}$
            \\
            \ref{eq:second_lag} & $7.54\times 10^{-4}$ & $1.86\times 10^{-1}$
            \\
            \bottomrule
        \end{tabular}
        \label{tab:maxerrors_energ}
    \end{table}
\end{center}
Since we have a manufactured solution, we can calculate the numerical and the real free energy $E$ at every time level according to \cref{eq:free_energy}, with energy parameters $\gamma_a = -\frac{6.3}{\ln(0.01)} \approx 1.368$, $\gamma_\ell=0$ and $\gamma_{a\ell}=0$, which is in accordance to the capillary pressure $p_c = \frac{6.3}{\ln(0.01)}\ln(s_a)$, so that at the continuous level, we have $\nu_a = -p_c$. At the discrete level, this is enforced through \cref{eq:nu_discrete}.
In \Cref{fig:energy_errors}, we show the relative errors in the free energy obtained with all the methods. 
We see that for both time steps, in general, the free energy is approximated more accurately with the second-order methods. Moreover, \textbf{MP} produces the lowest errors. 
Furthermore, in \Cref{tab:maxerrors_energ}, we observe that the maximum errors are the lowest with the midpoint method. For $\tau^n = 1$, the time-lagging schemes produce errors that are one order of magnitude larger than the fully implicit methods. We remark that even though for $\tau^n=1$ the $L^2$ errors in $s_a$ with \textbf{MP} are not the smallest ones, the midpoint method still approximates the free Gibbs energy more accurately than all other methods. This numerically demonstrates the relevance of the energy balance property \cref{eq:energy_discrete} that the midpoint method satisfies.

\begin{figure}
    \begin{subfigure}[t]{.45\textwidth}
        \centering
        \includegraphics[width=\linewidth]{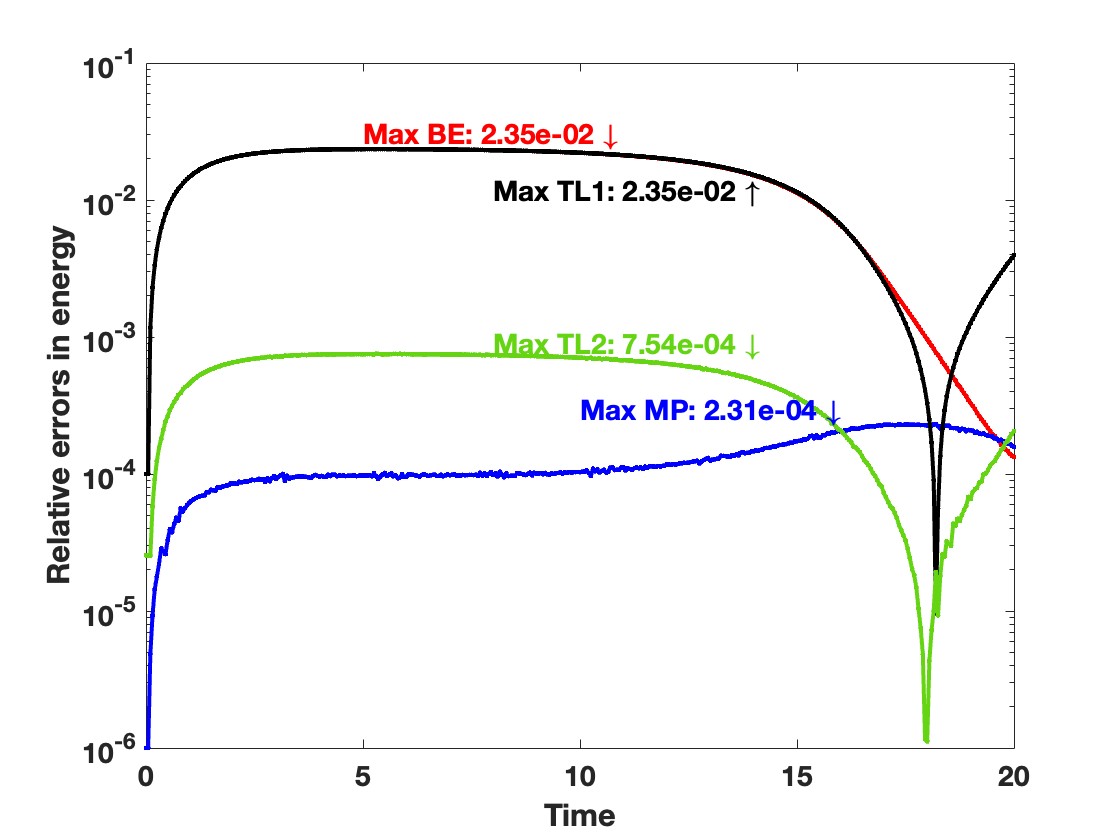}
        \caption{The relative errors in the free energy with $\tau^n = 0.05$.}
        \label{subfig:energy_0p1}
    \end{subfigure}
    \begin{subfigure}[t]{.45\textwidth}
        \centering
        \includegraphics[width=\linewidth]{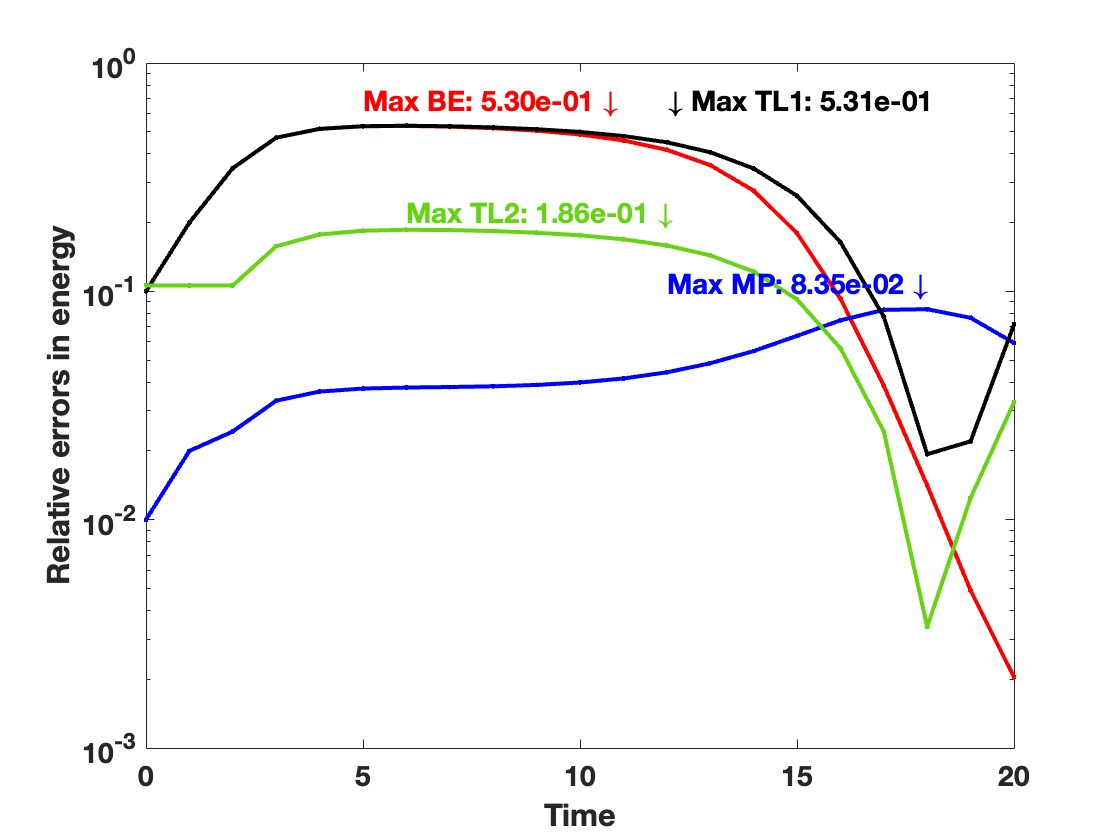}
        \caption{The relative errors in the free energy with $\tau^n = 1$.}
        \label{subfig:energy_2}
    \end{subfigure}
    \caption{The relative error in the free energy $E$ given in \cref{eq:free_energy} for test case in \cref{subsec:longtime}. Blue line: \textbf{MP}, red line: \textbf{BE}, black line: 
    \ref{eq:first_lag}, green line: 
    \ref{eq:second_lag}.}
    \label{fig:energy_errors}
\end{figure}


\subsection{Quarter-five spot problem}
 \label{subsec:q5spot}
	Finally, we consider a quarter-five spot problem, see e.g. \cite{EpshteynRiviere}. 
 The domain is $\Omega = [0,100]^2 \backslash \{[0,5]^2 \cup [95,100]^2\} \text{ m}^2$, i.e., it is the square $[0,100]^2 \text{ m}^2$ where two corners have been cut out; see \Cref{fig:domain_qfvs}.
	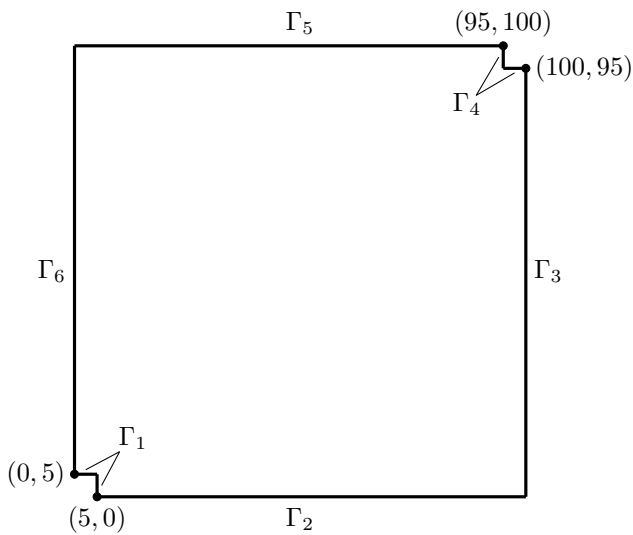
\begin{figure}
		\centering
		\begin{tikzpicture}[scale=0.30]
			\draw[black,very thick] (0,1) -- (1,1); 
			\draw[black,very thick] (1,1) -- (1,0); 
			\draw[black, very thick] (1,0) -- (20,0); 
			\draw[black, very thick] (20,0) -- (20,19); 
			\draw[black, very thick] (20,19) -- (19,19); 
			\draw[black, very thick] (19,19) -- (19,20); 
			\draw[black, very thick] (19,20) -- (0,20); 
			\draw[black, very thick] (0,20) -- (0,1); 
			\filldraw[black] (1,0) circle (5pt) node[anchor=north]{$(5,0)$};
			\filldraw[black] (0,1) circle (5pt) node[anchor=east]{$(0,5)$};
			\filldraw[black] (20,19) circle (5pt) node[anchor=west]{$(100,95)$};
			\filldraw[black] (19,20) circle (5pt) node[anchor=south]{$(95,100)$};
			\node at (2.6, 2.6) {$\Gamma_1$};
			\draw[black, thin] (0.5, 1.2) -- (2.,2.);
			\draw[black, thin] (1.2, 0.5) -- (2.,2.);
			\node at (10, -1) {$\Gamma_2$};
			\node at (21, 10) {$\Gamma_3$};
			\node at (17.4, 17.4) {$\Gamma_4$};
			\draw[black, thin] (19.5, 18.8) -- (17.8, 17.8);
			\draw[black, thin] (18.8, 19.5) -- (17.8, 17.8);
			\node at (10, 21) {$\Gamma_5$};
			\node at (-1, 10) {$\Gamma_6$};
		\end{tikzpicture}
		\caption{Depiction of domain $\Omega$ and different parts of the boundary, for example in \cref{subsec:q5spot}}
		\label{fig:domain_qfvs}
	\end{figure}
\\
	The boundary conditions are
	\begin{align*}
		\text{Inflow on } \Gamma_1\text{:} & \,\,\,p_\ell = 3\times 10^5 \text{ Pa}, \,\, s_a = 0.7,
		\\
		\text{Outflow on } \Gamma_4\text{:} & \,\,\,p_\ell = 10^5 \text{ Pa}, \,\, \lambda_a\kappa \nabla p_{c}\cdot \boldsymbol{n} = 0,
		\\
		\text{No-flow on } \Gamma_2 \cup \Gamma_3 \cup \Gamma_5 \cup \Gamma_6\text{:} & \,\,\,\lambda_\ell \kappa \nabla p_\ell \cdot \boldsymbol{n} = 0,\,\,\lambda_a \kappa \nabla  p_{c} \cdot \boldsymbol{n} = 0,
	\end{align*}
with the 
initial conditions 
$s_a = 0.2$ and $p_\ell = 10^5 \text{ Pa}$. The rest of the flow properties are given by:
	\begin{equation*}
		\phi = 0.2, \quad \mu_\ell = 2\times 10^{-3}\text{ Pa}\cdot\text{s}, \quad \mu_a = 5\times 10^{-4} \text{ Pa}\cdot\text{s} , 
	\end{equation*}
 and the relative permeabilities and capillary pressure follow a Brooks--Corey model \cite{BrooksCorey}:
	\begin{equation*}
		\kappa_\ell = (1-s_a)^2(1-s_a^{5/3}), \quad \kappa_a = s_a^{11/3}, \quad p_{c} = 5\times 10^3 s_a^{-1/3}.
	\end{equation*}
	We consider a heterogeneous medium where $\kappa = 5\times 10^{-8} \text{ m}^2$ 
 everywhere, except for in the square $[25,50]\times[25,50]$ where $\kappa$ is 1,000 times lower.
We remark that since the Brooks--Corey model does not follow the structure in \cref{eq:nu_a_definition_continuous}, this capillary pressure model is not thermodynamically consistent with the Gibbs free energy \cref{eq:free_energy}. For this case, the energy $E$ associated with the system would have to satisfy $-p_c = \nu_a \equiv \partial_{s_a}E - \partial_{s_\ell}E$. Therefore, in this example, we will only compare the quality of the discrete aqueous saturation.

  All the time discretizations use the same spatial discretization, which is based on a discontinuous Galerkin method with piecewise quadratic polynomials. 
The mesh consists of $38 \times 38$ quadrilateral elements. The final time for the simulation is $T=750$ seconds (12.5 minutes). 
We consider different constant time steps. 
(i) For $\tau = 1$, the only method that can successfully complete the simulation is \textbf{MP}, which performs in average 4.5 iterations per time step. 
(ii) For $\tau = 0.5$, both \textbf{MP} and \textbf{BE} complete the simulation with an average number of iterations per time step equal to 2.4 and 2.6, respectively. The two time lagging schemes 
\ref{eq:first_lag} and 
\ref{eq:second_lag} blow up for $\tau = 1$ and $\tau = 0.5$. 
(iii) Finally, for $\tau = 0.25$, all methods complete the simulation. 
The implicit methods \textbf{MP} and \textbf{BE} require an average of 2.02 and 1.96 iterations per time step, respectively. We summarize this in \Cref{tab:results_Q5S}.
Since the spatial discretization is the same for all four methods, the size of the problem is the same regardless of the value of $\tau$ and the time discretization. Therefore, the computational cost of the \textbf{MP} method for $\tau = 1$, should be comparable to the computational cost of 
\ref{eq:first_lag} and 
\ref{eq:second_lag} with $\tau = 0.25$. This is because the cost of performing 4 iterations per time step $\tau = 1$, is comparable to the cost of performing just one subiteration with a time step of $\tau/4 = 0.25$.
\begin{center}
\begin{table}[tbp]
    \centering
    \small
    \caption{Convergence summary for test case in \cref{subsec:q5spot} with different time steps. For \textbf{MP} and \textbf{BE}, we show average number of iterations per time step when the method converges.}
    \begin{tabular}{ccccc}
        \toprule
        $\tau$ & \textbf{MP} & \textbf{BE} & \ref{eq:first_lag} 
        & \ref{eq:second_lag}
        \\
        \midrule
        1 & 4.5 & Blows up &  Blows up &  Blows up
        \\
        0.5 & 2.4 & 2.6 &  Blows up &  Blows up
        \\
        0.25 & 2.02 & 1.96 &  Converges &  Converges
        \\
        \bottomrule
    \end{tabular}
    \label{tab:results_Q5S}
\end{table}
\end{center}
In \Cref{fig:sol_snapshots} we show snapshots of $s_a$ at different times obtained with \textbf{MP} with $\tau = 1$ second. Note that this result is what is expected from similar setups in the literature, see e.g. \cite{HDG_Riviere}. Moreover, in \Cref{fig:sol_diagonal} we show $s_a$ along the diagonal $x=y$ at different times. We
  compare all four methods with the largest $\tau$ for which they complete the simulation, i.e., we show the solution for \textbf{MP} with $\tau = 1$, \textbf{BE} with $\tau = 0.5$, and 
  \ref{eq:first_lag} and 
  \ref{eq:second_lag} with $\tau = 0.25$. We see that all four solutions are very similar. 
\begin{figure}
    \begin{subfigure}[t]{.45\textwidth}
        \centering
        \includegraphics[width=\linewidth]{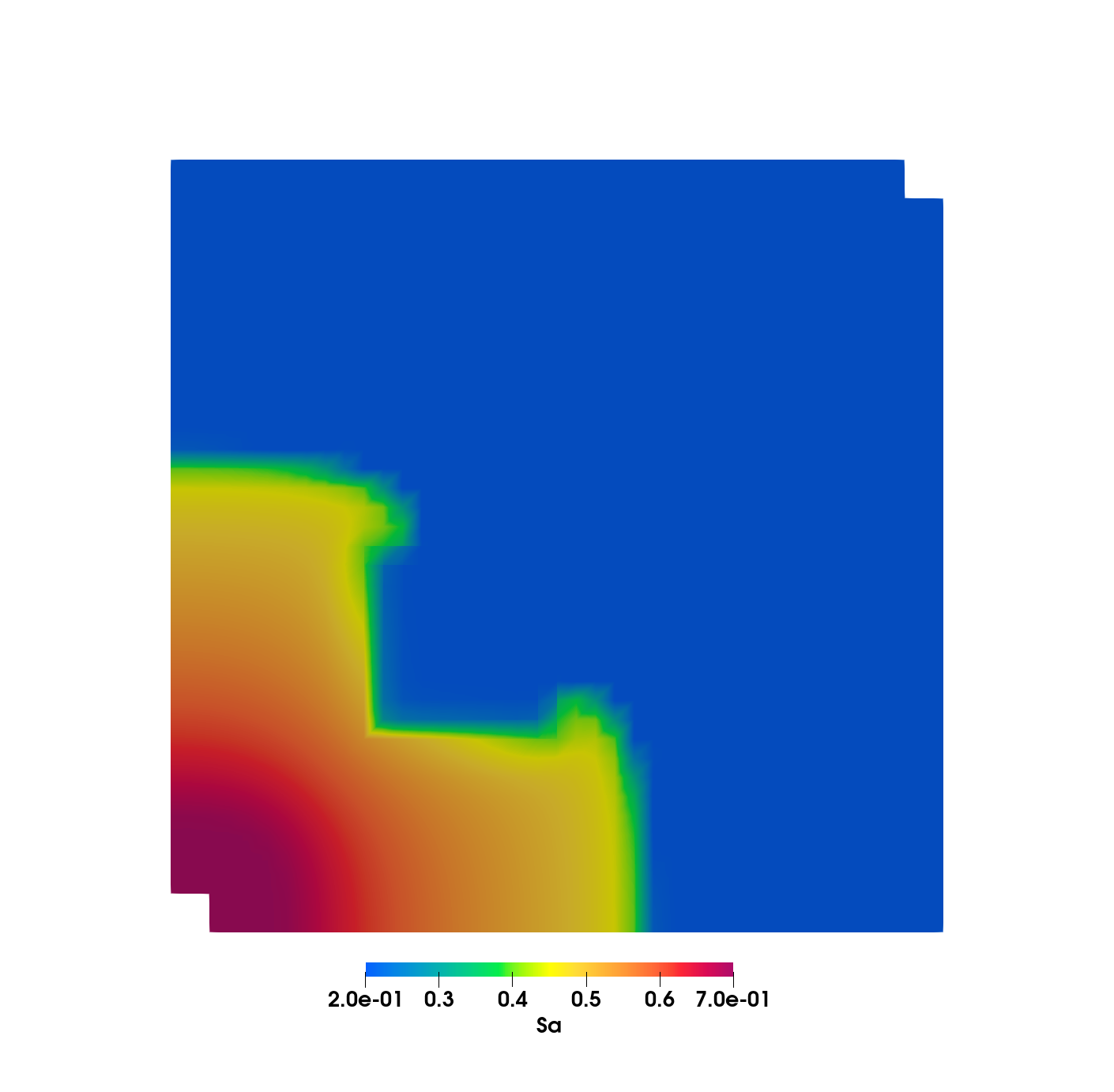}
        \caption{$t=250$ seconds.}
    \end{subfigure}
    \begin{subfigure}[t]{.45\textwidth}
        \centering
        \includegraphics[width=\linewidth]{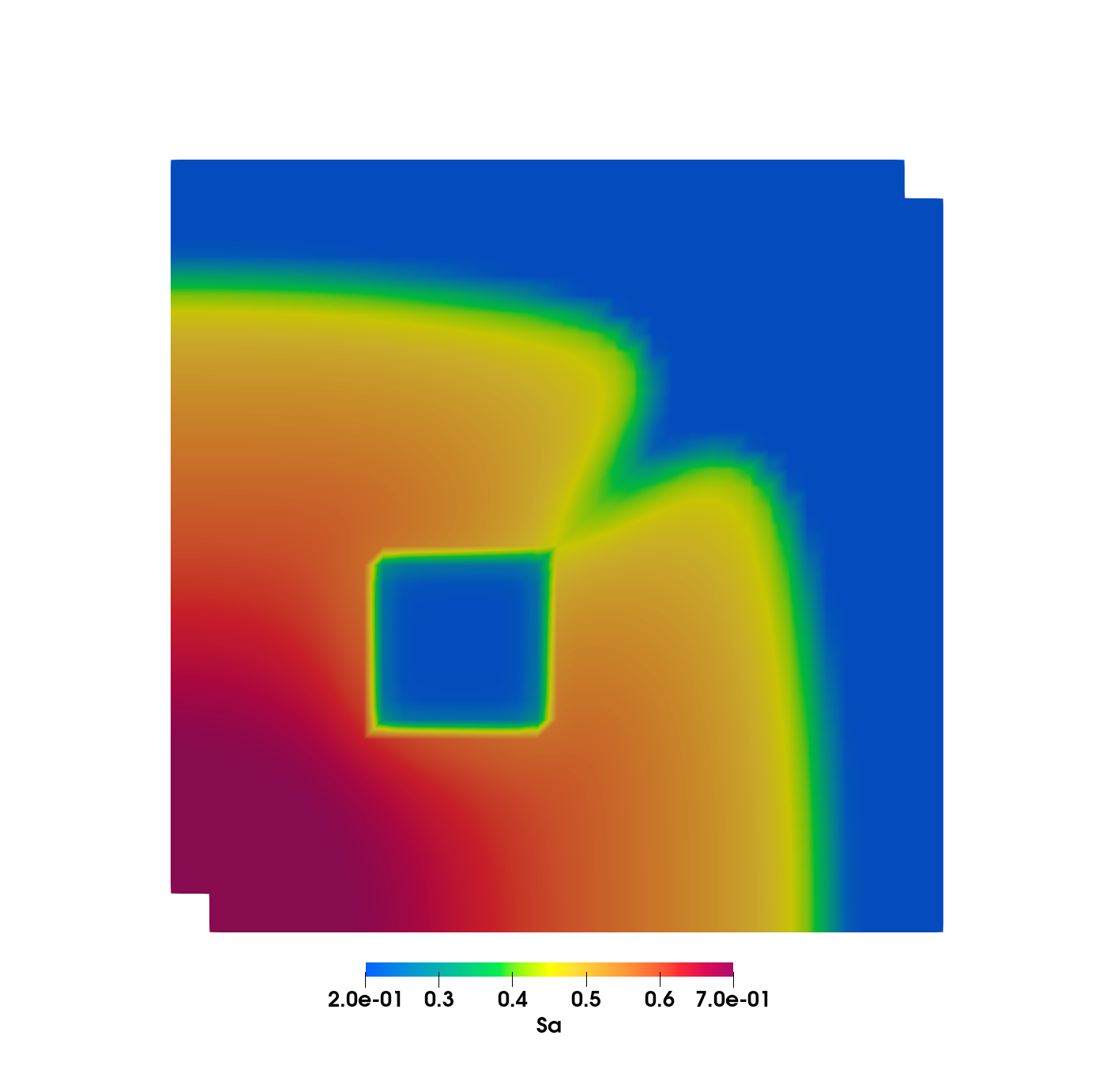}
        \caption{$t=500$ seconds.}
    \end{subfigure}
    %
    %
    \begin{subfigure}[t]{\textwidth}
        \centering
        \includegraphics[width=0.5\linewidth]{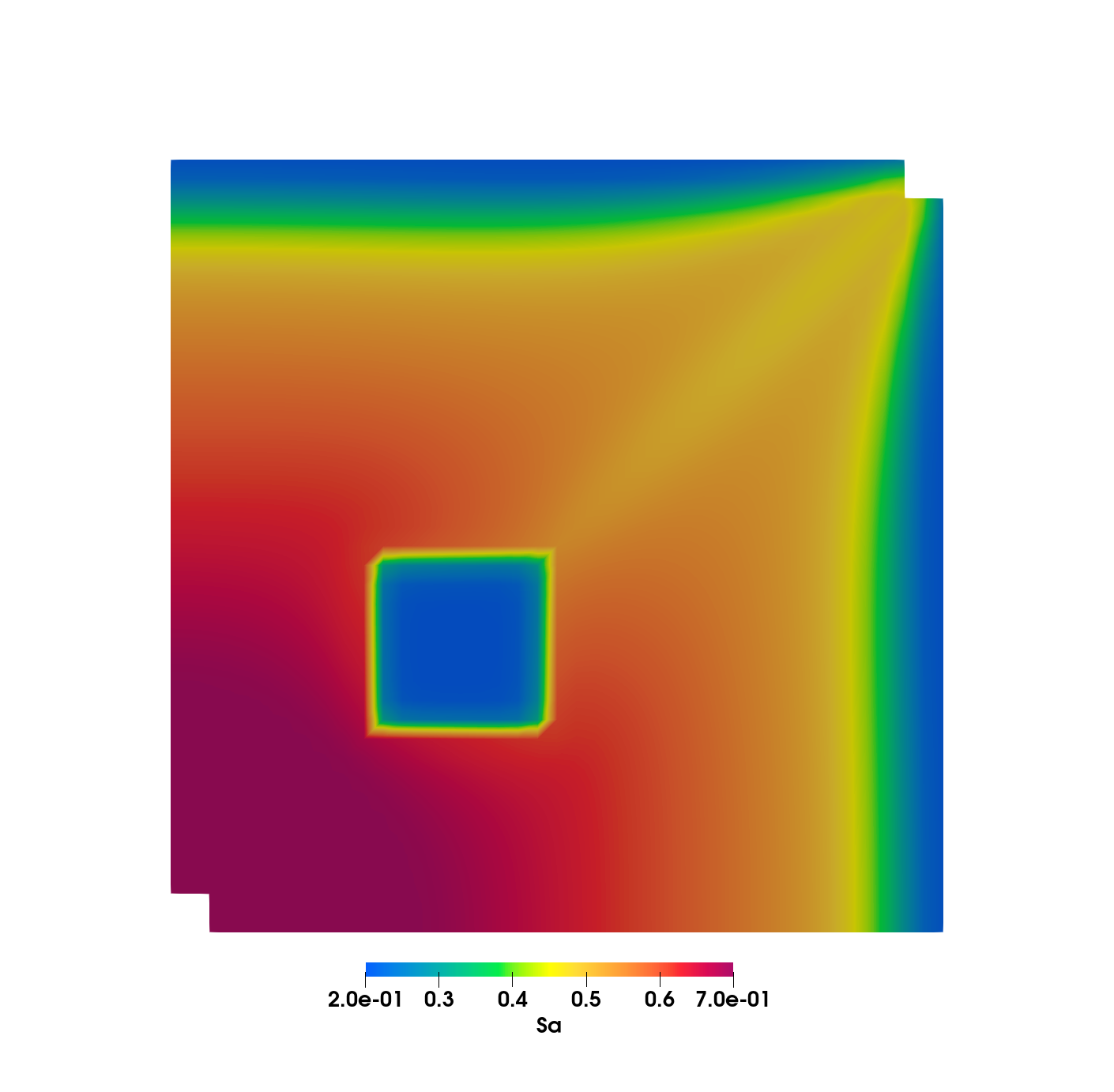}
        \caption{$t=750$ seconds.}
    \end{subfigure}
    
    \caption{Snapshots of the aqueous saturation at different times for test case in \cref{subsec:q5spot} calculated with the midpoint method \textbf{MP} with $\tau=1$ second.}
\label{fig:sol_snapshots}
\end{figure}
 \begin{figure}
    \begin{subfigure}[t]{.45\textwidth}
        \centering
        \includegraphics[width=\linewidth]{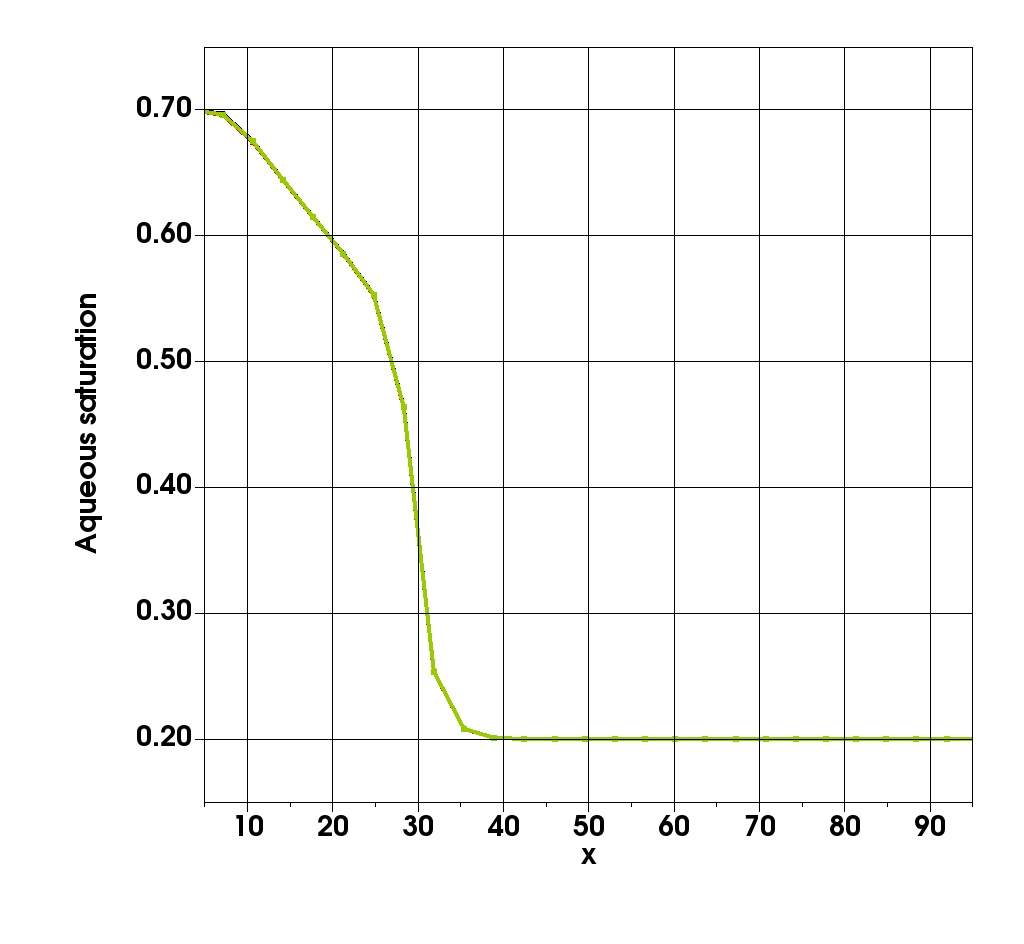}
        \caption{$t=250$ seconds.}
    \end{subfigure}
    \begin{subfigure}[t]{.45\textwidth}
        \centering
        \includegraphics[width=\linewidth]{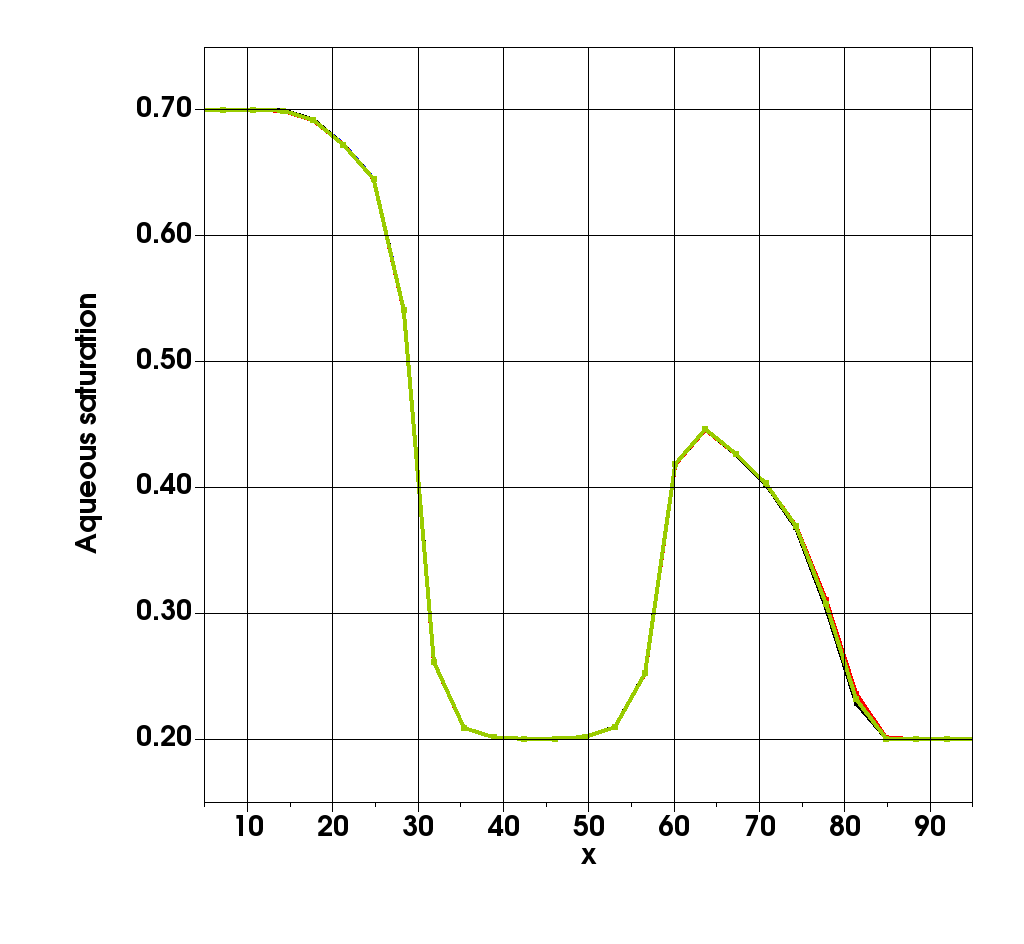}
        \caption{$t=500$ seconds.}
    \end{subfigure}
%
    %
    \begin{subfigure}[t]{.45\textwidth}
        \centering
        \includegraphics[width=\linewidth]{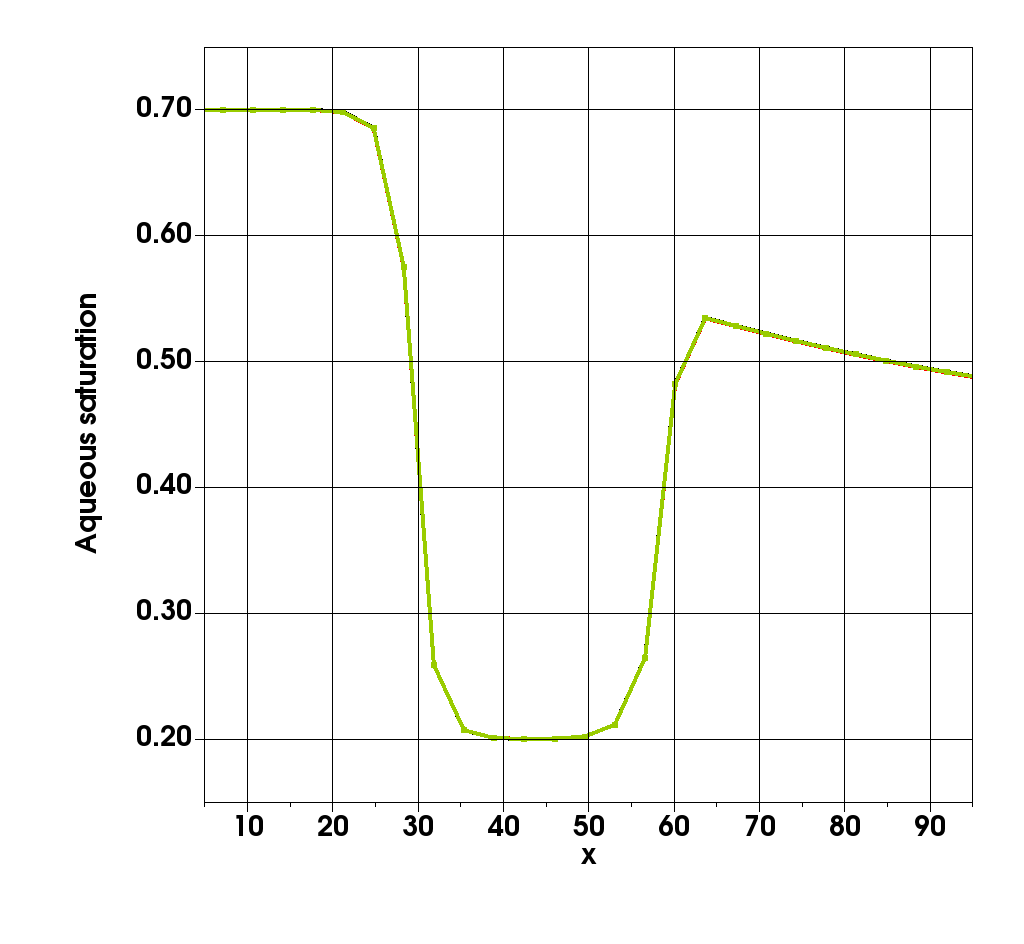}
        \caption{$t=750$ seconds.}
    \end{subfigure}
    \begin{subfigure}[t]{.45\textwidth}
        \centering
        \includegraphics[width=\linewidth]{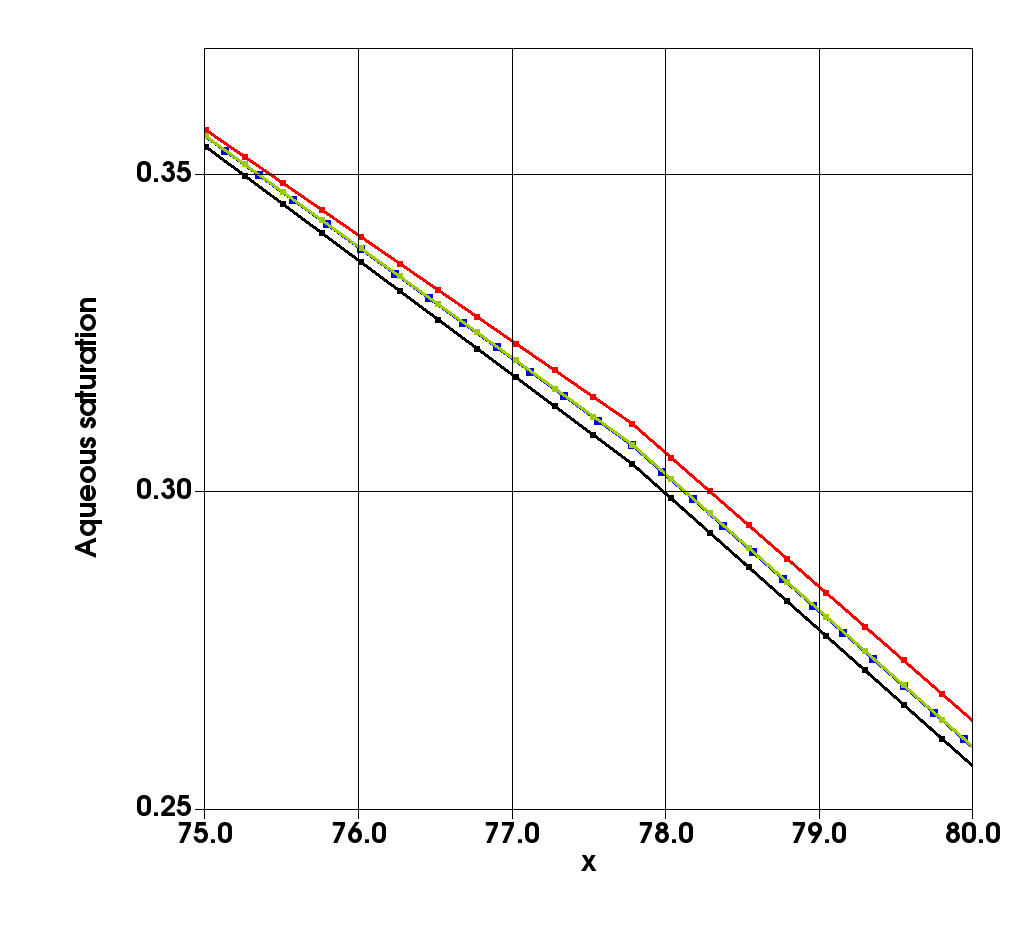}
        \caption{Close-up for $x \in [75 \,\,80]$ at $t = 500$ seconds.}
    \end{subfigure}
    \caption{Aqueous saturation along line $x=y$ for test case in \cref{subsec:q5spot}. Blue line: \textbf{MP} with $\tau=1$, red line: \textbf{BE} with $\tau=0.5$, black line: 
    \ref{eq:first_lag} with $\tau=0.25$, green line: 
    \ref{eq:second_lag} with $\tau=0.25$.}
\label{fig:sol_diagonal}
\end{figure}
 
	\section{Conclusions}
	\label{sec:conclusions}
	 We have presented a second-order time-stepping scheme for the two-phase problem in porous media. The method is proven to satisfy a discrete energy balance equation that is analogous to the energy balance satisfied at the continuous level. 
  We proved the convergence of the subiterations in the implicit step of the time-stepping scheme. 
  We compared this midpoint method against a backward Euler method and two time-lagging schemes. The theoretical rates of convergence are confirmed 
  for of all the methods considered. The midpoint method outperforms the other methods in terms of errors for a large final time, as well as for the approximation of the free energy. 
 This occurs both for the values of time step in the asymptotic regime, as well as for larger time steps. Future work includes extension of this method for the three-phase black oil problem, as well as time-step adaptivity.
\bibliographystyle{siamplain}
	\bibliography{references.bib,database2.bib}
\end{document}